\documentclass[a4paper]{article}
\usepackage[latin1]{inputenc} 
\usepackage[T1]{fontenc} 
\usepackage{RR,RRthemes}
\usepackage{hyperref}
\RRdate{Mai 2011}

\usepackage{amsmath,amsfonts,latexsym,amsthm,amsbsy,amssymb}
\usepackage[numbers,sort&compress]{natbib}
\usepackage{psfig,float} 
\usepackage{epsfig}
\usepackage{a4wide}
\usepackage{times}
                             
\newcommand{\nc}{\newcommand}
\newcommand{\rnc}{\renewcommand}
     
\nc{\KK}{\mbox{\rm I$\!$K}}
\nc{\dsp}{\displaystyle}
\nc{\bequ}{\begin{equation}}
\nc{\eequ}{\end{equation}}
\nc{\barr}{\begin{array}}
\nc{\earr}{\end{array}}

\nc{\Cm}{\Bbb C}
\nc{\Zm}{\Bbb Z}

\nc{\xy}{\mbox{$(x,y)$}}
\nc{\uv}{\mbox{$(u,v)$}}
\nc{\argsh}{\mbox{\rm argsh}}
\nc{\argch}{\mbox{\rm argch}}
\nc{\argth}{\mbox{\rm argth}}
\nc{\dilog}{\mbox{\rm dilog}}
\nc{\Ci}{\mbox{\rm Ci}}
\nc{\Si}{\mbox{\rm Si}}
\nc{\Ei}{\mbox{\rm Ei}}
\nc{\ch}{\mbox{\rm ch}}
\nc{\tgh}{\mbox{\rm th}}
\nc{\sh}{\mbox{\rm sh}}
\nc{\xt}{\widetilde{x}}
\nc{\xyz}{\mbox{$(x,y,z)$}}
\nc{\derpar}[2]{\frac{\partial #1}{\partial #2}}
\rnc{\d}[2]{\dps\frac{\mbox{d} #1}{\mbox{d} #2}}
\nc{\tend}[2]{\dps \lim_{#1 \rightarrow #2}}
\nc{\rdeu}{\mbox{\rm I$\!$R$^2$ }}
\nc{\RR}{\mbox{\rm I$\!$R}}
\nc{\NN}{\mbox{\rm I$\!$N}}
\nc{\rdeuplus}{\mbox{\rm I$\!$R$^2_{+}$}}
\nc{\quest}{\em}                          
\nc{\info}{\tt}
\nc{\beq}{\begin{equation}}
\nc{\eeq}{\end{equation}}
\nc{\soitf}[1]{{Soit $#1$ la fonction d\accent19efinie par :}}
\nc{\bfq}[1]{{\bf (#1) ---}}
\nc{\cc}{{\cal C}}
\nc{\Dprim}{{\cal D}^{'}}
\nc{\DD}{{\cal D}}
\nc{\eps}{\varepsilon}
\nc{\dps}{\displaystyle}
\newtheorem{exo}{\large \bf {Exercice}}
\nc{\nexo}{\begin{exo} \mbox{} \end{exo}}
\nc{\hK}{\hat K}
\nc{\hf}{\hat f}
\nc{\hvh}{\hat v_h}
\newtheorem{theo}{Theorem}

\newtheorem{Cor}{Corollary}

\newtheorem{rem}{Remark}

\nc{\fleche}{\mbox{
\LARGE $ \rule[1.3mm]{3.2cm}{.011in}\hspace*{-1.1cm}\longrightarrow$}
}
\nc{\norm}{\vec{n}}
\nc{\Es}{\widetilde{E}^s}
\nc{\Ex}{\widetilde{E}_x}
\nc{\Ev}{\rm \bf E}
\nc{\Evtil}{{\rm \bf \widetilde{E}}}
\nc{\ex}{\widetilde{e}_x}
\nc{\Estar}{E_x^\star}
\nc{\Hs}{\widetilde{H}^s}
\nc{\Ef}{{\cal E}}
\nc{\Som}{{\cal S}}
\nc{\Eff}{\widetilde{\cal E}}
\nc{\Hf}{{\cal H}}
\nc{\dv}{\mbox{ div }}
\nc{\dvg}{\mbox{ div}_\Gamma}
\nc{\rtg}{\vec{\mbox{rot}}_\Gamma}
\nc{\G}{{\cal G}}
\nc{\J}{{\cal J}}
\nc{\K}{{\cal K}}
\nc{\Ah}{\mbox{ \boldmath $A_h$}}
\nc{\Ih}{\mbox{ \boldmath $I_h$}}

\nc{\rot}{\mbox{\rm rot}}
\nc{\rotv}{{\bf \rm \rot}}
\nc{\kv}{\vec{k}}
\nc{\vv}{\vec{v}}
\nc{\modk}{\left| \kv \right|}

\nc{\lscal}{\left(  \hskip-1.2pt   \left(}
\nc{\rscal}{\right)  \hskip-1.2pt  \right)_h}

\nc{\Pt}{\varphi}
\nc{\Vt}{\mbox{\boldmath $\psi$}}

\newcommand{\norma}[1]{\left\|#1\right\|}
\def\div{\mathop{\mathrm{div}}\nolimits}
\def\grad{\mathop{\boldsymbol{\mathrm{grad}}}\nolimits}
\def\bV{\boldsymbol{V}}
\def\bVs{\boldsymbol{V}^{\star}}
\def\Vs{V^{\star}}
\def\Ps{P^{\star}}

\def\disc{\mathrm{disc}}
\def\L{\mathrm{L}}

\def\cC{\mathcal{C}}

\def\div{\mathop{\mathrm{div}}\nolimits}
\def\grad{\mathop{\boldsymbol{\mathrm{grad}}}\nolimits}
\def\bV{\boldsymbol{V}}
\def\bVs{\boldsymbol{V}^{\star}}
\def\Vs{V^{\star}}
\def\Ps{P^{\star}}

\def\disc{\mathrm{disc}}
\def\L{\mathrm{L}}

\def\cC{\mathcal{C}}

\RRauthor{
Eliane B\'ecache \thanks{eliane.becache@inria.fr, POEMS (CNRS:UMR 7231 - ENSTA - INRIA Rocquencourt),  INRIA, Domaine de Voluceau-Rocquencourt, BP
   105, F-78153 Le Chesnay C\'edex}%
  \and
Andr\'es Prieto\thanks{andres.prieto@usc.es, Departamento de Matem\'atica Aplicada, 
Universidade de Santiago de Compostela, 15782 Santiago de Compostela, Spain}%
}
\authorhead{B\'ecache \& Prieto}

\RRtitle{Remarques sur la stabilit\'e des PMLs Cart\'esiennes dans les coins}
\RRetitle{Remarks on the stability  of Cartesian PMLs in corners}
\titlehead{Stability  of Cartesian PMLs in corners}

\RRresume{Ce travail est une contribution \`a la compr\'ehension de la question de la stabilit\'e
 des couches absorbantes parfaitement adapt\'ees (PMLs) dans les coins, aux niveaux continu et discret.
 Des r\'esultats de stabilit\'e sont d'abord pr\'esent\'es pour des PMLs Cart\'esiennes associ\'ees \`a un syst\`eme 
hyperbolique de premier ordre g\'en\'eral. Puis, dans le cadre de la formulation du premier ordre pression-vitesse de 
l'\'equation des ondes acoustiques, une formulation non splitt\'ee des PMLs est discr\'etis\'ee par des \'el\'ements finis 
mixtes spectraux en espace et des diff\'erences finies en temps. On montre, \`a travers l'analyse de stabilit\'e 
de deux sch\'emas,  comment un mauvais choix pour la discr\'etisation en temps peut d\'et\'eriorer la condition 
de stabilit\'e CFL. Ces r\'esultats de stabilit\'e sont illustr\'es par des exp\'eriences num\'eriques.  }

\RRabstract{This work is a contribution to the understanding of the question of
stability of Perfectly Matched Layers (PMLs)  in corners, at continuous 
and discrete levels. First,  stability results are presented for the 
Cartesian PMLs associated to a general first-order hyperbolic system.
Then, in the context of the pressure-velocity formulation of the acoustic
wave propagation, 
an unsplit PML formulation is discretized
 with spectral mixed finite elements in space and finite differences
in time. 
It is shown,  through the stability analysis of two different  schemes, 
how a bad choice of the time discretization can deteriorate the CFL
stability condition. Some numerical results are finally presented to
illustrate these stability results. }

\RRmotcle{PML, couches absorbantes parfaitement adapt\'ees, stabilit\'e, \'el\'ements finis, diff\'erences finies, condition CFL}

\RRkeyword{PML, perfectly matched layers, absorbing layers, stability, finite elements, finite differences, CFL condition}
 \RRprojet{Poems}
\RRdomaine{1}

\RRthemeProj{poems}

 \RCParis

\begin{document}
\bibliographystyle{plain}

\RRNo{7620}

\makeRR

\section{Introduction}
The Perfectly Matched Layers (PML) technique, introduced in 
1994 by B\'erenger \cite{Ber94} for electromagnetic problems, is
considered as an efficient tool to simulate numerically wave propagation
problems stated in unbounded domains. In fact, during the last decade, 
this technique has been intensively applied in a wide range of areas
(acoustic and electromagnetic problems 
\cite{Rappaport:1995,Zhao:1996,pezhca98,Petropoulos:2000,
Collino:1998:1,Collino:1998:2},
elastodynamics
\cite{Col-Tso2,BFJ:2003,BasuChopra:pml-elas:04,ApKr06},
aeroacoustics
\cite{Hesthaven:1998,Hu:1996,Hu:2001,DJ-euler:03,Hag-pml03,
ApKr-pmleuler03,Nataf-JCP05,appelo2007pml}
among  others).

The mathematical analysis of the 
well-posedness and  stability of the continuous PML models (both 
in the original split formulation and in unsplit formulations) 
have been addressed in several works. Let us cite e.g. 
\cite{ab-got:97,Abarbanel:1999,rahm:phd,LVM:02} 
for the  well-posedness analysis and 
\cite{BJ:02,Hu:2001,BFJ:2003,BPG:04} for the  stability analysis.
It is now well-known that the original split  PML model,
as well as unsplit PML models, are (at least weakly) stable if
the original physical model is isotropic.
However, if the physical model is anisotropic then
the PML technique can lead to  unstable behaviors
(see \cite{BFJ:2003} for more details).

On the other hand, a natural question is the analysis of the stability of 
fully discrete schemes used for discretizing PMLs. In \cite{BJ:02}, 
this question has been considered for the discretization of isotropic 
Maxwell's equations. Precisely, it is shown that the Yee
scheme, applied for discretizing both  the  split and the unsplit PMLs, 
is stable under the standard CFL stability condition. 
This result is obtained for a layer in only one direction. When
considering Cartesian PMLs with layers in two (or three) directions, 
it is then natural to address the question of stability in corners. 
It is well known that this question is delicate for Absorbing Boundary 
Conditions (ABCs) and has been considered by several authors 
(e.g., \cite{Engq-Majd-77,Ramahi99}, and \cite{Giv-Hag-Pat:06} where 
long time instabilities for high order ABCs are mentioned).  

The main goal of the present paper is to study if instabilities could be
generated from the corner PML domains in 2D, as it is the case for some
ABCs in corners. The first result is that, on the continuous level,
PML corners are always stable (even in anisotropic models
assuming that the PML parameters for each direction are constant
and equal). On the discrete level, it is shown that
instabilities are related to an inadequate
time discretization of the auxiliary differential equations in the 
Cartesian PML formulation. In fact, in the context of the isotropic
acoustic model, we analyze two different time discretizations whose
CFL stability conditions are different only in the corner domains. 
Although both discrete scheme are consistent with the continuous model, 
only one has a CFL condition independent of the PML parameters, 
which coincides with the standard CFL of the scheme in the physical domain. 

Following is the outline of the paper. In Section~\ref{sec:2}, 
we describe the split B\'erenger's PML
formulation, written in Cartesian coordinates for a general first order 
hyperbolic system in two dimensions. We focus our attention on the 
equation system stated only in a corner domain. It is shown that, 
in contrast to a  layer in only one direction, PMLs are always stable 
in a corner, at least for a constant damping factor. In particular, 
it means that, even for anisotropic models, the split PMLs are stable 
at the continuous level. 

Section~\ref{sec:model} is devoted to the introduction of a model problem, 
the two dimensional wave equation, written as a first-order pressure-velocity 
system. The PML model considered here is the Zhao-Cangellaris unsplit formulation. 
Again, if the PML coefficients are assumed constant in the corner domain
(but possibly different), it is shown, via energy estimates, that the continuous 
model is stable.

The spatial semi-discretization, presented in Section~\ref{sec:approx-space}, 
is done with a spectral mixed finite element method based  on a quadrilateral 
mesh, used in~\cite{fauqueux}.

Section~\ref{sec:approx-time} is devoted to the time discretization, 
using explicit second-order finite differences. We first consider the 
scheme used in \cite{fauqueux}. A stability analysis of the scheme shows 
that the CFL condition  is deteriorated in the PML corner, when compared
to the CFL condition inside the fluid domain, the deterioration depending 
(in particular) on the damping factor. In order to avoid this, a new 
discretization in time in the PML corner is proposed, for which it is 
shown  that the CFL condition remains the same as in the fluid domain. 
These results are illustrated in Section~\ref{sec:num} with numerical 
simulations. 

\section{Stability of B\'erenger's splitted PMLs in corners for a
general hyperbolic system}\label{sec:2}

{\bf Notation.} Through the rest of the paper,  standard notations 
about functional Sobolev spaces 
are used without explicit definitions, 
$\norma{\cdot}_{L^2}$ and $(\cdot,\cdot)_{L^2}$ denote respectively the $L^{2}$-norm and the $L^{2}$-inner product.

In \cite{Col-Tso2}, the authors have shown how to design a splitted PML 
model, for a general first-order hyperbolic system. This construction 
has been applied in~\cite{BFJ:2003} for analyzing the well-posedness and 
the stability of a layer in one direction. In this section, we focus on 
the PML equations written in a corner domain. Following~\cite{Col-Tso2,BFJ:2003}, 
we consider the  first-order hyperbolic general Cauchy problem in the two 
dimensional free space:
\begin{align*}
& \partial_{t}  U-A_{x}\partial_{x}U-A_{y}\partial_{y}U=0,\\
&U(.,0)=U^{0},
\end{align*} 
where $U$ is a $m$ dimensional real-valued vector function 
$U:(x,y,t)\in \RR^2 \times \RR^+ \mapsto U(x,y,t) \in \RR^m$, 
$A_{x}$ and $A_{y}$  are $m \times m$ real-valued  symmetric matrices, 
and the initial data $U^{0}$ is a $m$ dimensional real-valued vector 
function $U^0: (x,y)\in \RR^2 \mapsto U^0(x,y) \in \RR^m$. 
It is then classical to show that the solution satisfies the energy conservation 
\[
\dps \frac{d}{dt}\norma{U}_{L^2}^2=0.
\]

We now introduce the splitted B\'erenger's PMLs equations in Cartesian
coordinates (\cite{Ber94,Col-Tso2}) using the splitting trick
for the corner domain: \\
{\it Find $(U^{x},U^{y})$ solution of the first-order hyperbolic system}
\begin{align}
\label{eq:pml_corner_1}& \partial_{t} U^{x}+\sigma_{x}U^{x}-A_{x}\partial_{x}U=0,
\\
\label{eq:pml_corner_2}&\partial_{t} U^{y}+\sigma_{y}U^{y}-A_{y}\partial_{y}U=0,
\\
\label{eq:pml_corner_3}&U=U^{x}+U^{y},
\\
\label{eq:pml_corner_4}& U^{x}(.,0)=(U^{x})^{0},\ U^{y}(.,0)=(U^{y})^{0}.
\end{align}

It is straightforward to verify the following
\begin{theo}
If $\sigma_{x}$ and $\sigma_{y}$ are constant and equal to $\sigma$ 
then the solution $(U^{x},U^{y})$ of system 
(\ref{eq:pml_corner_1})-(\ref{eq:pml_corner_4}) satisfies
\begin{equation}
\label{eq:nrj-1}
\frac{d}{dt}\norma{U}_{L^2}^2=-2\sigma\norma{U}_{L^2}^2\le 0.
\end{equation}
In this case, the PML corner is strongly stable in the sense that the
solution can be bounded by the initial data, uniformly in time:
\begin{equation}
\label{eq:est-1}
\exists C>0, \quad \left\| U (t)\right\|_{L^2} \le C \left\| U^0 \right\|_{L^2},
\quad \forall t >0. 
\end{equation}

\end{theo}
\begin{proof}
Since the absorbing functions are equal, adding equations 
(\ref{eq:pml_corner_1}) and (\ref{eq:pml_corner_2}), and using
(\ref{eq:pml_corner_3}), we see that $U$ satisfies the following
equation:
\begin{equation}
 \label{eq:pml-coins}
\partial_{t}U+ \sigma U-A_{x}\partial_{x}U-A_{y}\partial_{y}U=0.
\end{equation}
In this case, the PML corner appears as a ``classical'' dissipation
term. The energy identity (\ref{eq:nrj-1}) and the estimate (\ref{eq:est-1})
follow then easily.
\end{proof}

The strong stability shows in particular that the solution can not grow 
linearly in time, therefore there is no possible long-time instability 
coming from the corner, when the damping factor is the same in the two 
directions (contrarely to some high order ABCs \cite{Giv-Hag-Pat:06}).

The proof of an analogous energy estimate, for non-constant and different 
absorbing functions $\sigma_{x}$ and $\sigma_{y}$, remains open
(see, for instance~\cite{BJ:02}).

\begin{rem}
For several models as, for instance, Maxwell's equations or the acoustic 
scalar wave equation, 
several authors~\cite{Chew:1994,Collino:1998:1,Collino:1998:2,
Rappaport:1995,Turkel1998533} derived the B\'erenger's PML model in 
terms of a complex coordinate stretching in the frequency domain. 
In the time domain, 
this point of view leads to the construction 
of the so-called unsplit PMLs, which do not need the splitting of the unknowns 
fields,  
 but the introduction of new 
additional unknowns. The main advantage of this formulation is that it 
preserves the spatial differential operators and consequently the original
spatial discretizations developed for the physical models.
The original B\'erenger's PML can then be reformulated 
in several forms (e.g.~\cite{ZC,BPG:04,Gedney96,Abarbanel:2002}), 
depending on the choice of the auxiliary unknowns. Note that all these 
formulations are ``equivalent'',  in the sense that one can go from one set of unknowns to the other through elementary linear operations 
(see e.g.~\cite{HDR-Bec-2003,BJ:02} for Maxwell's equations).

But contrarily to the split PMLs, there is no way of designing unsplit 
PMLs (which preserve the original operator) for a general first-order 
hyperbolic system. However,  it is straightforward to see that in the 
particular case of a PML corner, with $\sigma_x= \sigma_y=\sigma$, 
the complex coordinate stretching leads to equation (\ref{eq:pml-coins}), 
i.e. to the same equation as the one obtained with the splitting. 
In that case actually there is neither splitting anymore nor additional 
unknowns.
\end{rem}  

\section{The model problem: unsplit PMLs for the scalar wave equation}
\label{sec:model}
In the rest of the paper, we focus on a model problem, 
the two-dimensional acoustic wave equation, written as a first-order 
pressure-velocity system, for which  we  illustrate the construction 
of the Cartesian unsplit PMLs at the corner and derive some energy 
estimates. The governing equations of the original first-order 
hyperbolic system in terms of pressure and velocity fields are given by
\begin{align}
&\frac{1}{\mu}\partial_{t}P=\div \bV,\label{eq:cont-sys_1}\\
&\rho\partial_{t}\bV=\grad P,\label{eq:cont-sys_2}
\end{align}
where $P$ is the pressure, $\bV=(V_{x},V_{y})$ is the velocity,
 $\mu$ and $\rho$ are the bulk modulus and the mass density respectively,
which are assumed to be  positive bounded functions. 
We denote by $c=\sqrt{\mu/\rho}$ the acoustic sound velocity. 
 Additionally, we should include
the initial data for $P$ and $V$ and boundary conditions for the pressure
field in system (\ref{eq:cont-sys_1})-(\ref{eq:cont-sys_2}), but in order 
to simplify the presentation, in the rest of the paper they will be 
systematically omitted. 
 
We introduce the Cartesian unsplit PMLs at the corner following the 
Zhao-Cangellaris' formulation \cite{ZC} (recall that the solution of 
the split Berenger's PML can be deduced by simple linear combinations 
from the solution of the unsplit 
formulation and conversely, 
see~\cite{BJ:02}):~\\
{\it Find $(P,\Ps,\bV,\bVs)$ such that}
\begin{align}
&\frac{1}{\mu}\partial_{t}\Ps=\div \bVs,\label{eq:pml-sys_1}\\
&\rho\left(\partial_{t}+\sigma_{x}I\right)V_{x}=\partial_{x}P,\\
&\rho\left(\partial_{t}+\sigma_{y}I\right)V_{y}=\partial_{y}P,\\
&\rho\partial_{t}\Vs_{x}=\rho\left(\partial_{t}+\sigma_{y}I\right)V_{x},\\
&\rho\partial_{t}\Vs_{y}=\rho\left(\partial_{t}+\sigma_{x}I\right)V_{y},\\
&\frac{1}{\mu}\partial_{tt}^{2}\Ps=
\frac{1}{\mu}\left(\partial_{t}+\sigma_{x}I\right)\left(\partial_{t}+\sigma_{y}I\right)P,\label{eq:pml-sys_6}
\end{align}
satisfying the adequate initial conditions, where 
$\bVs=(\Vs_{x},\Vs_{y})$ and $I$ is the identity operator. In the
following, we introduce the standard notation, for any arbitrary positive bounded
function $\nu$:
\[
(\phi,\psi)_{\nu}= (\nu \phi,\psi)_{L^2}, \quad 
\left\| p \right\|^2_{\nu} = (\nu p,p)_{L^2}.
\]

Using the same arguments showed for the Maxwell's equations in \cite{BJ:02}, 
straightforward computations lead to the following result:
\begin{theo}
If $\sigma_{x}$ and $\sigma_{y}$ are constant (but eventually different), the energy
\begin{equation*}
\mathcal{E}_2(t)=\frac12 \left\{ \left\|  \partial^{2}_{t}\Ps\right\|^2_{1/\mu} 
+ \left\| \partial^{2}_{t}\bVs \right\|^2_{\rho}
+ \left\| \sigma_{x} \partial_{t}  \Vs_{x}  \right\|^2_{\rho} 
+ \left\| \sigma_{y} \partial_{t}  \Vs_{y}  \right\|^2_{\rho} 
\right\},
\end{equation*}
satisfies
$$
\frac{d}{dt}\mathcal{E}_2(t)=-2\left(\sigma_{x} \left\|
    \partial_{t}^{2}\Vs_{x}\right\|^2_{\rho}
+ \sigma_{y} \left\|\partial_{t}^{2}\Vs_{y}\right\|^2_{\rho}\right)\le 0.
$$
\end{theo}

Finally, the result showed in \cite{BPG:04} can also be obviously extended
here:
\begin{theo}
If $\sigma_{x}= \sigma_{y}= \sigma $, where $\sigma$ is a positive
constant, we have the following identity:
\[
\frac12 \frac{d}{dt} \left\{ \left\| P  \right\|^2_{1/\mu}  +  \left\|
    \bVs\right\|^2_{\rho} + \sigma^2  \left\| F  \right\|^2_{1/\mu}
    \right\} = -2 \sigma \left\| P  \right\|^2_{1/\mu}, 
\]
where 
\[
F(t)= \dps \int_0^t P(s) ds.
\]
This implies in particular the decrease of 
the energy of order 0: 
\begin{equation*}
\mathcal{E}_0(t)=\frac12\left\{\left\| P  \right\|^2_{1/\mu}  +  \left\|
    \bVs\right\|^2_{\rho} \right\}.
\end{equation*}
\end{theo}

\section{Semi-discretisation in space using a mixed spectral element method}
\label{sec:approx-space}
The system (\ref{eq:pml-sys_1})-(\ref{eq:pml-sys_6}) is approximated in space 
with   
$Q_r-Q_r^{\disc}$ 
mixed spectral elements based on hexaedral meshes, described in~%
\cite{cohen-fauq:00,fauqueux}. For the sake of simplicity in its 
description, we first introduce the spatial discretization for the 
original acoustic model before applying it to the PML problem.

\subsection{Spatial approximation in the fluid domain}
In order to describe briefly the mixed finite elements, we first
consider the approximation of the equations set in the fluid domain $\Omega$
with Dirichlet boundary conditions for the pressure field. The variational
formulation of (\ref{eq:cont-sys_1})-(\ref{eq:cont-sys_2}) is then: \\
{\it Find $P \in H^1_0(\Omega), \bV \in (L^2(\Omega))^2$ such that}
\begin{align*}
&\dps \frac{d}{dt} (P,\Pt)_{1/\mu}= - ( \bV, \grad \Pt)_{L^2}, 
\quad   \forall \Pt \in H^1_0(\Omega), \\
& \dps \frac{d}{dt}(\bV, \Vt )_{\rho} =(\grad P,  \Vt)_{L^2}, 
\quad \forall \Vt \in (L^2(\Omega))^2.
\end{align*}

We introduce ${\cal T}= \dps \cup_{i=1}^{N_E} K_i$ a partition of
$\Omega$ with $N_E$ 
quadrilateral elements, $\widehat{K}=[0,1]^2$ the unit element, and the
conform mappings $F_i$ such that $F_i(\widehat{K})=K_i$, $\forall
i=1,\ldots,N_E$. 
We set $DF_i$  the Jacobian matrix of $F_i$
and its Jacobian $J_i= \mbox{ det } DF_i$. We finally define the 
approximation spaces:
\begin{align*}
& {\cal P}_{h 0}^r=\left\{ \Pt\in \cC^0(\Omega), \quad 
\Pt_{|K_i} \circ F_i \in Q_r(\widehat{K}), \quad 
\Pt=0  \mbox{  on } \partial \Omega \right\},\\
& {\cal V}_h^r=\left\{ \Vt\in (\L^2(\Omega))^2, \quad 
|J_i| DF_i^{-1} \Vt_{|K_i} \circ F_i \in  (Q_r(\widehat{K}))^2 \right\},  
\end{align*}
with $Q_r(\widehat{K})$ the set of polynomials whose degree is less or 
equal to $r$ in each spatial variable. 
For both spaces, the 
interpolation points coincide with the Gauss-Lobatto quadrature points. 
The semi-approximate problem is then:\\
{\it Find $P_h \in {\cal P}_{0h}^r, \bV \in {\cal V}_h^r$ such that}
\begin{align*}
&\dps \frac{d}{dt} (P_h,\Pt_h)_{1/\mu}= - ( \bV_h, \grad \Pt_h)_{L^2}, 
\quad \forall \Pt_h \in {\cal P}_{0h}^r,\\
& \dps \frac{d}{dt}(\bV_h, \Vt_h )_{\rho} =(\grad P_h, \Vt_h)_{L^2}, 
\quad \forall \Vt_h \in {\cal V}_h^r.
\end{align*}
Let $\{\varphi_j:\,1\le j\le N_{\cal P}\}$ 
the finite element basis of ${\cal P}_{0h}^r$
   and $\{\Vt_m:\,1\le
m\le N_{\cal V}\}$ the basis functions of ${\cal V}_h^r$ 
(see~\cite{fauqueux,cohen-fauq:00} for a more 
detailed definition). These basis are associated to the interpolation
points that coincide with the quadrature points of the
Gauss-Lobatto quadrature formula which is exact for polynomials of
degree $2r-1$. We introduce the  discrete mass and stiffness
matrices 
\begin{eqnarray*}
&M_{ij}=\left( \varphi_{i},\varphi_{j}\right)_{1/\mu},
&\qquad 1\le i,j\le N_{\cal P},\\
&(R_{x})_{jm}=\left( \partial_{x}\varphi_{j},\Vt_{m}\right)_{L^2},
&\qquad 1\le j\le N_{\cal P},\ 1\le m\le N_{\cal V},\\
&(R_{y})_{jm}=\left( \partial_{y}\varphi_{j},\Vt_{m}\right)_{L^2},
&\qquad 1\le j\le N_{\cal P},\ 1\le m\le N_{\cal V},\\
&B_{ml}=\left( \Vt_{m},\Vt_{l}\right)_{\rho},
&\qquad 1\le m,l\le N_{\cal V}.
\end{eqnarray*}
where all the integrals are computed by using the Gauss-Lobatto
quadrature formula. The main advantages of this method are: 
 (a) it provides diagonal or block diagonal mass matrices (mass
lumping)~; (b) the stiffness matrices are independent of the mesh and
of the physical properties of the fluid medium (gain of storage).

The semi-discretized scheme can then be written in the matrix form:
\begin{align}
&\dps \frac{d}{dt} M P_h  +
R_{x}V_{x,h}+R_{y}V_{y,h} =0 , \label{sd-mat-1}\\
&
\dps \frac{d}{dt}BV_{x,h}=R_{x}^{*}P_{h},\\
&
\dps \frac{d}{dt}BV_{y,h}=R_{y}^{*}P_{h},\label{sd-mat-3}
\end{align}
where we identify the notations for the unknowns and their basis 
coordinates. 

\subsection{Spatial approximation in the PML corner}
Coming back to the semi-discretization of the PML corner problem, we
introduce some new matrices, defined for any $L^1$ positive function $\nu$
\begin{eqnarray*}
&M^{\nu}_{ij}=\left( \nu\varphi_{i},\varphi_{j}\right)_{1/\mu},&\qquad 1\le i,j\le N_{\cal P},\\
&B^{\nu}_{ml}=\left( \nu \Vt_{m},\Vt_{l}\right)_{\rho},&\qquad 1\le m,l\le N_{\cal V}.
\end{eqnarray*}
The semi-discretized scheme can then be written in the matricial form
as:
\begin{align}
&\dps \frac{d}{dt} M\Ps_{h} + R_{x}\Vs_{x,h}+R_{y}\Vs_{y,h}=0 ,\label{eq:pml-sd_1}\\
 &\dps \left(\frac{d}{dt}   B + B^{\sigma_x}\right)  V_{x,h}  =R_{x}^{*}P_{h},\\
 &\dps \left(\frac{d}{dt}  B  +  B^{\sigma_y}\right) V_{y,h} =R_{y}^{*}P_{h},\\
&\dps \frac{d}{dt} B\Vs_{x,h}= \left(\frac{d}{dt} B  + B^{\sigma_y}\right)V_{x,h},\\
&\dps \frac{d}{dt}B\Vs_{y,h}=\left(\frac{d}{dt}B  + B^{\sigma_x}\right)V_{y,h},\\
&\dps \frac{d^2}{dt^2} M\Ps_{h}= 
 \dps \left(\frac{d^2}{dt^2}  M + \frac{d}{dt} M^{\sigma_x+\sigma_y}+ 
  M^{\sigma_x \sigma_y}\right)    P_{h},\label{eq:pml-sd_6}
\end{align}

\section{Two alternatives for the time discretization}
\label{sec:approx-time}
We introduce $\Delta t$ the time step and $(P_{h})^{n}$,
$(\Ps_{h})^{n}$, $(\bV_{h})^{n+\frac12}$, $(\bVs_{h})^{n+\frac12})$ 
denote the spatial vector fields associated to the degrees of freedom of 
each unknown in the fully discrete problem at time $n\Delta t$ and 
$(n+1/2)\Delta t$, respectively. Since we work with the unsplit PML 
formulation, one notices that there is the second-order equation in time 
(\ref{eq:pml-sd_6}) to be discretized in the PML corner. This leads to 
two different second-order centered finite difference 
approximations in time. 

{\bf Notation.} In the sequel, $\norma{\cdot}$ and $(\cdot\left.\right|\cdot)$  
denote respectively the Euclidian norm and its associated inner product. 
For any positive matrix $S$, we introduce the notation:
$$
(U \left.\right|V)_S = (S U \left.\right| V) 
\quad ; \quad 
\norma{V}_S^2= (S V \left. \right|  V) 
$$

\subsection{Time Discretization in the fluid domain}
A centered second order finite difference scheme is used for the 
time discretization of matrix system (\ref{sd-mat-1})-(\ref{sd-mat-3}) :
\begin{align*}
  & M \frac{(P_{h})^{n+1}-(P_{h})^{n}}{\Delta t} +
 R_{x}(V_{x,h})^{n+1/2}+R_{y}(V_{y,h})^{n+1/2}=0 ,\\[12pt]
  &   B \frac{(V_{x,h})^{n+1/2}- (V_{x,h})^{n-1/2}}{\Delta t} 
    =R_{x}^{*}P_{h}^{n}, \\[12pt]
  &   B  \frac{(V_{y,h})^{n+1/2}- (V_{y,h})^{n-1/2}}{\Delta t} 
  =R_{y}^{*}P_{h}^n,
\end{align*}
completed with the adequate initial conditions. We recall that this 
scheme is stable under the CFL stability condition (e.g.~\cite{fauqueux,cohen-fauq:00}):
\begin{equation*}
\dps \sup_{u\ne 0} \frac{(M^{-1} Ku   \left. \right|  u)}{    \|    u\|^2} <
\frac{4}{\Delta t^2} \Longleftrightarrow M- \dps \frac{\Delta t^2}{4} K 
\mbox{ is positive definite},
\end{equation*}
where  $K$ is the stiffness matrix defined as $K= R{\cal B}^{-1} R^{*}$,
$R$ is the $N_{\cal P} \times 2 N_{\cal V}$  matrix defined as $R= (R_x \; R_y)$ 
and ${\cal B}$ is the $2N_{\cal V} \times 2 N_{\cal V}$ block diagonal 
matrix, each block of size $ N_{\cal V} \times N_{\cal V}$ being equal 
to $B$. Classically, on a regular grid, this condition is expressed as 
\begin{equation}
\label{eq:cfl-fluid-clas}
C\dps \frac{\Delta t}{h} <1,
\end{equation}
where $C$ is a constant independent of $h$ and $\Delta t$.

In a homogeneous fluid and using a regular mesh, this condition reduces to
\begin{equation*}
\dps \sqrt{\mu/\rho} \frac{\Delta t}{h} < cfl_{d,r} , 
\end{equation*}
where $d$ is the dimension of the space and $r$ is the degree of the
polynomials. The constant $cfl_{d,r}$ can be related to the CFL number 
in dimension one (see \cite{fauqueux}) by 
\begin{equation}
\label{eq:cfl-fluid-3}
\dps  cfl_{d,r}= \dps \frac{cfl_{1,r}}{\sqrt{d}} .
\end{equation}
In particular for $r=1$, we have $cfl_{1,1}=1$ and in dimension $d$ 
we recover the classical CFL condition:
\begin{equation*}
\dps c \frac{\Delta t}{h} \le \frac{1}{\sqrt{d}},
\end{equation*}
which is natural since the $Q_1 - Q_1^{\disc}$ discretization on a regular 
grid is equivalent to the standard second-order finite difference scheme.

\subsection{Time Discretization in the PML corner: { scheme A}}
In this section we consider the scheme which was originally used in \cite{fauqueux}
for the time discretization of the governing equations in the PML corners.
Let us first introduce some useful notations: 
\[
(D_{\Delta t} U)^k = (U^{k+1/2}-U^{k-1/2})/\Delta t \; ; \quad 
(I_{\Delta
  t} {U})^k=(U^{k+1/2}+U^{k-1/2})/2.
\]
The time approximation for (\ref{eq:pml-sd_1})-(\ref{eq:pml-sd_6}) 
provided by scheme A is written as:
\begin{align*}
&M(D_{\Delta t}\Ps_{h})^{n+\frac12}+R_{x}(\Vs_{x,h})^{n+\frac12}+R_{y}(\Vs_{y,h})^{n+\frac12}=0,\\
&B(D_{\Delta t} V_{x,h})^{n} +  B^{\sigma_x} (I_{\Delta t} V_{x,h})^{n}=R_{x}^{*}P_{h}^{n},\\
&B(D_{\Delta t} V_{y,h})^{n}+ B^{\sigma_y} (I_{\Delta t}V_{y,h})^{n} =R_{y}^{*}P_{h}^{n},\\
&B(D_{\Delta t}\Vs_{x,h})^{n}=B(D_{\Delta t} V_{x,h})^n + B^{\sigma_y}
(I_{\Delta t}V_{x,h})^n ,\\
&B(D_{\Delta t}\Vs_{y,h})^{n}=B(D_{\Delta t} V_{y,h})^n + B^{\sigma_x}
(I_{\Delta t} V_{y,h})^n,\\
&M(D_{\Delta t}^{2}\Ps_{h})^{n}=M(D_{\Delta t}^2
P_{h})^{n}+ M^{\sigma_x+\sigma_y}   (D_{\Delta t} I_{\Delta t}
{P_{h}})^{n}+    M^{\sigma_x \sigma_y}  P_{h}^{n}.
\end{align*}

The initial motivation of this study comes from the instabilities
observed in the numerical simulations for some values of the PML 
parameters (see Section~\ref{sec:num}). Their origin is located in the 
PML corner domains (both for variable and constant $\sigma_\alpha$, $\alpha=x,y$), 
and we also observed that these instabilities were removed when decreasing
either $\Delta t$ or the values of $\sigma_\alpha$. 

Since the question of
stability for variable $\sigma_\alpha$ is an open question even for the
continuous PML models, we will assume for the analysis of the  schemes 
that $\sigma_\alpha$ are positive constants for $\alpha=x,y$ in the corner domains. 
Since $\sigma_x$ is constant, we have $B^{\sigma_x}= \sigma_x B$ and 
we can define the discrete centered operator 
$(D_{\Delta t}^{\sigma_x} U)^k=(D_{\Delta t} U)^k + \sigma_x (I_{\Delta t} {U})^k$
to approximate $\partial_t + \sigma_x I$. With these notations and
assumptions, the scheme can be rewritten as:
\begin{align}
&M(D_{\Delta t}\Ps_{h})^{n+\frac12}+R_{x}(\Vs_{x,h})^{n+\frac12}+R_{y}(\Vs_{y,h})^{n+\frac12}=0,\\
&B(D_{\Delta t}^{\sigma_x} V_{x,h})^{n}=R_{x}^{*}P_{h}^{n},\\
&B(D_{\Delta t}^{\sigma_y} V_{y,h})^{n}=R_{y}^{*}P_{h}^{n},\\
&B(D_{\Delta t}\Vs_{x,h})^{n}=B(D_{\Delta t}^{\sigma_y} V_{x,h})^n,\\
&B(D_{\Delta t}\Vs_{y,h})^{n}=B(D_{\Delta t}^{\sigma_x} V_{y,h})^n,\\
&M(D_{\Delta t}^{2}\Ps_{h})^{n}=M\left((D_{\Delta t}^2
P_{h})^{n}+(\sigma_x+\sigma_y) (D_{\Delta t} I_{\Delta t}
{P_{h}})^{n}+\sigma_x \sigma_y P_{h}^{n}\right).\label{eq:pml-sys-discr-A_6}
\end{align}
We can  show the stability result:
\begin{theo}\label{th:A}
Assume that  $\sigma_{x}$ and $\sigma_{y}$ are constant and equal to $\sigma$. 
The discrete scheme A is stable if the matrix
\begin{equation}\label{eq:cond_A}
M-\dps \frac{\Delta t^2}{4} K_\sigma 
\end{equation}
is positive definite, where $K_\sigma=K+\sigma^2 M$ 
and $K= R{\cal B}^{-1} R^{*}$.
\end{theo}
\begin{proof}
Since the absorbing functions are equal, it is possible to rewrite the
scheme as
\begin{align}
&M(D_{\Delta t}\Ps_{h})^{n+\frac12}=
 R_{x}(\Vs_{x,h})^{n+\frac12}+R_{y}(\Vs_{y,h})^{n+\frac12},\label{eq:pml-sys-discr-A_1-1}\\
&B  (D_{\Delta t}\Vs_{x,h})^{n}=R_{x}^{*}P_{h}^{n},\label{eq:pml-sys-discr-A_1-2}\\
&B(D_{\Delta t}\Vs_{y,h})^{n}=R_{y}^{*}P_{h}^{n},\label{eq:pml-sys-discr-A_1-3}\\
&M(D_{\Delta t}^{2}\Ps_{h})^{n}=M\left((D_{\Delta t}^2
P_{h})^{n}+ 2 \sigma (D_{\Delta t} I_{\Delta t}
{P_{h}})^{n}+ \sigma^2 P_{h}^{n}\right).\label{eq:pml-sys-discr-A_1-4}
\end{align}
Applying $D_{\Delta t} $ to (\ref{eq:pml-sys-discr-A_1-1}), it is then
easy to eliminate $\bVs_h$ by using 
(\ref{eq:pml-sys-discr-A_1-2}), (\ref{eq:pml-sys-discr-A_1-3}), 
and (\ref{eq:pml-sys-discr-A_1-4}), and rewrite a scheme only in terms 
of $P_{h}$:
\[
M\left( (D_{\Delta t}^2
P_{h})^{n}+ 2 \sigma (D_{\Delta t} I_{\Delta t}
{P_{h}})^{n}+ \sigma^2 P_{h}^{n}\right)+ R {\cal B}^{-1} R^{*} P_{h}^{n}=0.
\]
Introducing the stiffness matrix 
$K_\sigma = R {\cal B}^{-1} R^{*} + \sigma^2 M$, 
we see that the following discrete energy 
\[
\mathcal{E}_{\mathrm{A},h}^{n+\frac12}=
\frac12\left\{ \left\| \frac{P_{h}^{n+1}- P_{h}^{n}}{\Delta
      t}\right\|_M^2 + (K_\sigma P_{h}^{n}    \left. \right| P_{h}^{n+1})
\right\}
\]
satisfies the identity
\[
\dps \frac{ \mathcal{E}_{\mathrm{A},h}^{n+\frac12}-
  \mathcal{E}_{\mathrm{A},h}^{n-\frac12}}{\Delta t}= -2 \sigma \left\| 
(D_{\Delta t} I_{\Delta t}{P_{h}})^{n} \right\|_M^2,
\]
and that 
$\mathcal{E}_{\mathrm{A},h}^{n}\ge 0$ if matrix (\ref{eq:cond_A})
is positive definite. 
\end{proof}

Let us remark that the sufficient condition stated in the above Theorem for the stability of the numerical scheme depends on the value
of the absorbing function $\sigma$, which means that  the CFL condition is not the same in the PML
corner than in the fluid domain. More
precisely,  if ${C}$ is the constant appearing in the CFL
condition in the fluid domain (see (\ref{eq:cfl-fluid-clas})), then the
stability condition in the PML corner can be expressed as
\begin{equation}
\label{eq:cfl-pml-A}
\Delta t^2 \left( \frac{{C}^2}{h^2} + \frac{\sigma^2}{4}\right)  <1
\Longleftrightarrow  \frac{{C} \Delta t}{h} < \frac{1}{\left(1+
  \dps \frac{\sigma^2 h^2}{4 {C}^2} \right)^{1/2}},
\end{equation}
which is more restrictive than the one in the fluid domain. 

\paragraph{Application: homogeneous medium and approximation with
second order finite differences in time and in space.} 
We consider here an homogeneous medium and we use the lowest order mixed
spectral elements for the approximation, i.e. $r=1$, on a regular grid. 
The scheme can then be written as a finite difference scheme in space and time,
\begin{align*}
&\frac{1}{\mu}(D_{\Delta t}\Ps_{h})^{n+\frac12}_{ij}=
(D_{\Delta x}\Vs_{x,h})^{n+\frac12}_{ij}+(D_{\Delta y}\Vs_{y,h})^{n+\frac12}_{ij},\\
&\rho(D_{\Delta t}^{\sigma_{x}} V_{x,h})^{n}_{i+\frac12 j}=(D_{\Delta x}P_{h})^{n}_{i+\frac12 j},\\
&\rho(D_{\Delta t}^{\sigma_{y}} V_{y,h})^{n}_{i j+\frac12}=(D_{\Delta y}P_{h})^{n}_{i j+\frac12},\\
&(D_{\Delta t}\Vs_{x,h})^{n}_{i+\frac12 j}=(D_{\Delta t}^{\sigma_{y}} V_{x,h})^n_{i+\frac12 j},\\
&(D_{\Delta t}\Vs_{y,h})^{n}_{i j+\frac12}=(D_{\Delta t}^{\sigma_{x}} V_{y,h})^n_{i j+\frac12},\\
&((D_{\Delta t})^{2}\Ps_{h})^{n}_{ij}=(((D_{\Delta t})^2
P_{h})^{n}+(\sigma_x+\sigma_y) (D_{\Delta t} I_{\Delta t}
{P_{h}})^{n}+\sigma_x \sigma_y P_{h}^{n})_{ij},
\end{align*}
where the subscripts $ij$ denote the degrees of freedom at nodes of 
coordinates $(i\Delta x,j\Delta y)$ as usual in the finite difference 
discretizations in the spatial variables. 

The scheme in the fluid corresponds to the classical second order
finite difference scheme. In the two-dimensional case, 
its CFL stability condition in the fluid domain corresponds to 
take ${C}= \sqrt{2}c$ in (\ref{eq:cfl-fluid-clas}) and so, we obtain
\begin{equation}
\label{eq:cfl-pml-B-cp}
c \Delta t < \dps \frac{h}{\sqrt{2}}.
\end{equation}
In the PML corner, this condition becomes
\begin{equation}
\label{eq:cfl-pml-A-cp}
c \Delta t < \dps \frac{h}{\sqrt{2}} \frac{1}{\left(1+
  \dps \frac{\sigma^2 h^2}{8c^2} \right)^{1/2}}.
\end{equation}

\subsection{Time discretization in the PML corner: { scheme B}}
In this section we propose a new scheme which corresponds to another
time discretization of the last second order differential equation
(\ref{eq:pml-sd_6}):
\begin{align*}
&M(D_{\Delta t}\Ps_{h})^{n+\frac12}+R_{x}(\Vs_{x,h})^{n+\frac12}+R_{y}(\Vs_{y,h})^{n+\frac12}=0,\\
&B(D_{\Delta t} V_{x,h})^{n} +  B^{\sigma_x} (I_{\Delta t} V_{x,h})^{n}=R_{x}^{*}P_{h}^{n},\\
&B(D_{\Delta t} V_{y,h})^{n}+ B^{\sigma_y} (I_{\Delta t}V_{y,h})^{n} =R_{y}^{*}P_{h}^{n},\\
&B(D_{\Delta t}\Vs_{x,h})^{n}=B(D_{\Delta t} V_{x,h})^n + B^{\sigma_y}
(I_{\Delta t}V_{x,h})^n ,\\
&B(D_{\Delta t}\Vs_{y,h})^{n}=B(D_{\Delta t} V_{y,h})^n + B^{\sigma_x}
(I_{\Delta t} V_{y,h})^n ,\\
&M(D_{\Delta t}^{2}\Ps_{h})^{n}=M(D_{\Delta t}^2
P_{h})^{n}+ M^{\sigma_x+\sigma_y}   (D_{\Delta t} I_{\Delta t}
{P_{h}})^{n}+    M^{\sigma_x \sigma_y}  (I_{\Delta t}^2 P_{h})^{n}.
\end{align*}
For constant values of the damping parameters, this scheme becomes:
\begin{align}
&M(D_{\Delta t}\Ps_{h})^{n+\frac12}+R_{x}(\Vs_{x,h})^{n+\frac12}+R_{y}(\Vs_{y,h})^{n+\frac12}=0,
\label{eq:pml-sys-discr-B_1}\\
&B(D_{\Delta t}^{\sigma_{x}} V_{x,h})^{n}=R_{x}^{*}P_{h}^{n},\label{eq:pml-sys-discr-B_2}\\
&B(D_{\Delta t}^{\sigma_{y}} V_{y,h})^{n}=R_{y}^{*}P_{h}^{n},\label{eq:pml-sys-discr-B_3}\\
&B(D_{\Delta t}\Vs_{x,h})^{n}=B(D_{\Delta t}^{\sigma_{y}} V_{x,h})^n,\label{eq:pml-sys-discr-B_4}\\
&B(D_{\Delta t}\Vs_{y,h})^{n}=B(D_{\Delta t}^{\sigma_{x}} V_{y,h})^n,\label{eq:pml-sys-discr-B_5}\\
&M(D_{\Delta t}^{2}\Ps_{h})^{n}=M(D_{\Delta t}^{\sigma_{x}} D_{\Delta t}^{\sigma_{y}} P_{h})^{n},
\label{eq:pml-sys-discr-B_6}
\end{align}
In this scheme, the operator $(\partial_{t}+\sigma_x I)(\partial_{t}+\sigma_y I)$ 
appearing in (\ref{eq:pml-sd_6}) has been approximated by 
$D_{\Delta t}^{\sigma_{x}} D_{\Delta t}^{\sigma_{y}}$, which is in
 some sense more natural than the discretization used in (\ref{eq:pml-sys-discr-A_6}). 

We can  show the stability result: 
\begin{theo}\label{th:B}
We assume that $\sigma_x$ and $\sigma_y$ are constants (not
necessarily equal). The discrete scheme B is stable if the matrix
\begin{equation}\label{eq:cond_B}
{\cal B}- \dps \frac{\Delta t^2}{4}R^{*}M^{-1} R 
\end{equation}
is positive definite.
\end{theo}
\begin{proof}
We apply $D_{\Delta t}^2I_{\Delta t}$ to (\ref{eq:pml-sys-discr-B_1}) and
multiply with $D_{\Delta t}^2 \Ps_{h}$:
 \[
 \underbrace{(D_{\Delta t}I_{\Delta t} M D_{\Delta t}^2\Ps_{h}   \left. \right|  D_{\Delta
     t}^2 \Ps_{h})}_{(\mathrm{I})}+ \underbrace{(D_{\Delta t}^2I_{\Delta t} \bVs_h   \left. \right| 
     R^{*} D_{\Delta
     t}^2 \Ps_{h})}_{(\mathrm{II})}=0.
 \]
We develop each term $(\mathrm{I})$ and $(\mathrm{II})$ and separately:
\[
(\mathrm{I})= \frac{1}{2\Delta t} \left((D_{\Delta
     t}^2 \Ps_{h})^{n+1}-(D_{\Delta
     t}^2 \Ps_{h})^{n-1}  \left. \right|  (D_{\Delta
     t}^2 \Ps_{h})^{n}\right)_M, 
\]
and, since $\sigma_x$ and $\sigma_y$ are constants, using (\ref{eq:pml-sys-discr-B_6}) 
and then (\ref{eq:pml-sys-discr-B_2})-(\ref{eq:pml-sys-discr-B_5}), we obtain
\[\begin{array}{lcl}
(\mathrm{II}) &=& (D_{\Delta t}^2I_{\Delta t} \bVs_h  \left. \right|     D_{\Delta t}^{\sigma_x}
D_{\Delta t}^{\sigma_y}   R^{*} P_{h})\\[12pt]
&=& (D_{\Delta t}^2I_{\Delta t} \Vs_{x,h}   \left. \right|     D_{\Delta t}^{\sigma_x}
D_{\Delta t}^{\sigma_y}  B D_{\Delta t}^{\sigma_x}  V_{x,h} ) + 
(D_{\Delta t}^2I_{\Delta t} \Vs_{y,h}  \left. \right|     D_{\Delta t}^{\sigma_x}
D_{\Delta t}^{\sigma_y}  B D_{\Delta t}^{\sigma_y}  V_{y,h} )
\\[12pt]
&=& (D_{\Delta t}^2I_{\Delta t} \Vs_{x,h}  \left. \right|     (D_{\Delta t}^{\sigma_x})^2
   D_{\Delta t}  \Vs_{x,h} )_B + 
(D_{\Delta t}^2I_{\Delta t} \Vs_{y,h}  \left. \right|     
(D_{\Delta t}^{\sigma_y})^2   D_{\Delta t}  \Vs_{y,h} )_B.
\end{array}
\]
The last two terms in $(\mathrm{II})$ can be rewritten as a difference
of norms. For instance, the first one leads to
\[
\begin{array}{lcl}
(D_{\Delta t}^2I_{\Delta t} \Vs_{x,h} \left. \right|     (D_{\Delta t}^{\sigma_x})^2
   D_{\Delta t}  \Vs_{x,h} )_B&=& (D_{\Delta t}^2I_{\Delta t} \Vs_{x,h}  \left. \right| 
   D_{\Delta t}^3 \Vs_{x,h} )_B + 2 \sigma_x \left\| D_{\Delta
   t}^2I_{\Delta t} \Vs_{x,h}\right\|_B^2\\[12pt]
& & + \sigma_x^2 (D_{\Delta  t}^2I_{\Delta t} \Vs_{x,h}  \left. \right| D_{\Delta t}I_{\Delta t}^2
   \Vs_{x,h})_B\\[12pt]
&=& \dps \frac{1}{2\Delta t} \left( 
\left\| (D_{\Delta  t}^2  \Vs_{x,h})^{n+\frac12} \right\|^2_B - \left\| (D_{\Delta  t}^2  \Vs_{x,h})^{n-\frac12} \right\|^2_B
\right)\\[12pt]
&& + 2 \sigma_x \left\| (D_{\Delta
   t}^2I_{\Delta t} \Vs_{x,h})^n\right\|_B^2\\[12pt]
&& \hspace*{-0.5cm} + \dps \frac{\sigma_x^2}{2\Delta t} \left(  
\left\| (I_{\Delta t} D_{\Delta  t}  \Vs_{x,h})^{n+\frac12} \right\|^2_B - \left\| (I_{\Delta t} D_{\Delta  t}  \Vs_{x,h})^{n-\frac12}\right\|^2_B.
\right)
\end{array}
\]
Hence, if the discrete energy is defined by
\begin{equation}
\label{eq:discrete-energy-B}
\begin{array}{lcl}
\mathcal{E}_{\mathrm{B},h}^{n+\frac12}&=& \dps \frac12\left\{
\left((D_{\Delta t}^2 \Ps_{h})^{n+1}  \left. \right|   (D_{\Delta t}^2
  \Ps_{h})^{n}\right)_M 
+ \left\| (D_{\Delta  t}^2  \bVs_h)^{n+\frac12} \right\|^2_B \right.\\[12pt]
&& \dps \left. +
\sigma_x^2 \left\| (I_{\Delta t} D_{\Delta  t}  \Vs_{x,h})^{n+\frac12}
\right\|^2_B + \sigma_y^2 \left\| (I_{\Delta t} D_{\Delta  t}  \Vs_{y,h})^{n+\frac12}
\right\|^2_B
\right\},
\end{array}
\end{equation}
then we have the identity
\[
\frac{\mathcal{E}_{\mathrm{B},h}^{n+\frac12}-\mathcal{E}_{\mathrm{B},h}^{n-\frac12}
  }{\Delta t}= - 2 \sigma_x \left\| (D_{\Delta
   t}^2I_{\Delta t} \Vs_{x,h})^n\right\|_B^2- 2 \sigma_y \left\| (D_{\Delta
   t}^2I_{\Delta t} \Vs_{y,h})^n\right\|_B^2 \le 0.
\]
In order to point out the stability condition we finally rewrite the
discrete energy (\ref{eq:discrete-energy-B}) as 
\[
\begin{array}{l} \dps 
\mathcal{E}_{\mathrm{B},h}^{n+\frac12}=\frac12\left\{
\left\| (I_{\Delta t} D_{\Delta t}^2 \Ps_{h})^{n+\frac12} \right\|^2_M
- \frac{\Delta t^2}{4}  \left\| ( D_{\Delta t}^3 \Ps_{h})^{n+\frac12} \right\|^2_M
+ \left\| (D_{\Delta  t}^2  \bVs_h)^{n+\frac12} \right\|^2_B \right.\\[12pt]
\left. +
\sigma_x^2 \left\| (I_{\Delta t} D_{\Delta  t}  \Vs_{x,h})^{n+\frac12}
\right\|^2_B + \sigma_y^2 \left\| (I_{\Delta t} D_{\Delta  t}  \Vs_{y,h})^{n+\frac12}
\right\|^2_B
\right\}.\end{array}
\]
Now, applying $D_{\Delta t}$ to (\ref{eq:pml-sys-discr-B_1}), it holds
\[
M ( D_{\Delta t}^3 \Ps_{h})^{n+\frac12} = -R (D_{\Delta t}^2
\bVs_h)^{n+\frac12}, 
\]
and so
\begin{multline*} 
\dps \mathcal{E}_{\mathrm{B},h}^{n+\frac12}=\frac12\left\{
\left\| (I_{\Delta t} D_{\Delta t}^2 \Ps_{h})^{n+\frac12} \right\|^2_M
+ ( (B -\frac{\Delta t^2}{4} R^{*} M^{-1}R)  (D_{\Delta  t}^2  \bVs_h)^{n+\frac12} \left. \right|  (D_{\Delta  t}^2  \bVs_h)^{n+\frac12})  \right.\\
\left. +
\sigma_x^2 \left\| (I_{\Delta t} D_{\Delta  t}  \Vs_{x,h})^{n+\frac12}
\right\|^2_B + \sigma_y^2 \left\| (I_{\Delta t} D_{\Delta  t}  \Vs_{y,h})^{n+\frac12}
\right\|^2_B
\right\}.
\end{multline*}
Finally, since the matrix (\ref{eq:cond_B}) is assumed positive definite, 
we obtain $\mathcal{E}_{\mathrm{B},h}^{n+\frac12}\ge 0$ and conclude 
the stability of the scheme.
\end{proof}

\begin{Cor}
If the discrete scheme B is used, then the CFL stability condition 
(\ref{eq:cond_B}) holds in the PML corner and in the fluid domain.
\end{Cor}
\begin{proof}
We recall that the CFL condition in the fluid domain is given by requiring
$M-\frac{\Delta t^2}{4} R {\cal B}^{-1} R^{*}$ is definite positive
whereas the stability condition in the PML corner is obtained
by assuming ${\cal B}-\frac{\Delta t^2}{4} R^{*} M^{-1}R$ is definite positive.

If we denote by $(\lambda_j,v_j)$, $j=1,\ldots, N_{\cal P}$ the 
positive eigenvalues and the eigenvectors of problem
\begin{equation}
\label{eq:vp-1}
 R {\cal B}^{-1} R^{*} v_j=\lambda_j Mv_j,\qquad v_{j}\ne 0,
\end{equation}
the CFL condition in the fluid domain can be expressed as 
$\max_j \lambda_j < \frac{\Delta t^2}{4}$. 
Analogously, if we denote by $(\mu_j,w_j)$, $j=1,\ldots,2N_{\cal V}$ the eigenvalues and the eigenvectors of problem
\begin{equation}
\label{eq:vp-2}
 R^{*} {M}^{-1} R w_j=\mu_j {\cal B}w_j,\qquad w_{j}\ne 0,
\end{equation}
the CFL condition in the PML corner can be expressed as 
$\max_j \mu_j < \frac{\Delta t^2}{4}$. 
Then, it is easy to show that $\lambda_j$ is a nonzero eigenvalue of 
(\ref{eq:vp-1}) if and only if it is also an eigenvalue of (\ref{eq:vp-2}).
In fact, let us assume that $(\lambda_j,v_j)$ is an eigenvalue and 
eigenvector associated to (\ref{eq:vp-1}). If we multiply (\ref{eq:vp-1}) 
by ${\cal B}^{-1}R^*M^{-1}$ and define $w_j={\cal B}^{-1}R^*v_j$, we have
${\cal B}^{-1}R^*M^{-1}Rw_j=\lambda_j w_j$ (obviously $w_{j}\ne 0$ since
$\lambda_{j}$ and $v_{j}$ are not null). Hence, it is clear that 
$\lambda_j$ is also an eigenvalue of problem (\ref{eq:vp-2}).
Reciprocally, we can use analogous arguments to show the equivalence 
between both eigenvalue problems. Hence, $\max_j \mu_j= \max_j \lambda_j$
and consequently the two CFL conditions coincide. 
\end{proof}

\section{Numerical illustration}
\label{sec:num}
We consider the free propagation of waves generated by a compact supported 
initial condition (Ricker impulse, e.g. \cite{fauqueux}) for the pressure 
field centered at the point source $(17,17)$, a central frequency equal 
to $1$ and a spectral ratio equal to $0.5$. The computational domain is 
$[-2, 20]^2$ and the thickness of the Cartesian PML is $2$ (so that 
the physical domain is $[ 0, 18]^2$). 
The physical parameters are $\rho=1$, $\mu=1$.

The aim of this section is to illustrate numerically the sufficient 
conditions stated in Theorems \ref{th:A} and \ref{th:B}
with two different spatial discretizations. We consider a spectral finite 
element method based on $Q_r-$Lagrange rectangular piecewise continuous 
elements for the pressure fields $P$ and $\Ps$ 
and piecewise discontinuous elements for the velocity fields $\bV$ and $\bVs$
on uniform grids. The numerical experiments are performed with $r=1$ and $r=5$. 

Remind that the theoretical CFL condition has been established in the case 
of a constant absorbing function. In the following experiments, we 
first  try to recover these results numerically for constant damping functions: 
\begin{equation}
\label{sigma}
\sigma_{x} (x)= \begin{cases}
0 &  \mbox{if}  \quad x \in [0,18] \\
\sigma & \mbox{otherwise}
\end{cases} \quad ; 
 \quad
\sigma_{y} (y)= \begin{cases}
0 &\mbox{if}  \quad y \in [0,18] \\
\sigma & \mbox{otherwise}
\end{cases} .
\end{equation}
Then, we consider a quadratic damping function, i.e. we replace the constant value $\sigma$ in 
(\ref{sigma}) by the quadratic profile which is  continuous at the inner boundary of the PML
 and whose upper bound is denoted by $\sigma^{*}$.

\paragraph{Second-order scheme and constant absorbing function.} 
For $r=1$, the positive definite condition (\ref{eq:cond_B})  for 
scheme B yields to the standard CFL condition (\ref{eq:cfl-pml-B-cp}),
which coincides with the original CFL condition for the wave equation 
and then is independent of the values of $\sigma_{x}$ and  $\sigma_{y}$. 
For scheme A, the CFL condition is expected to depend on the value of 
the absorbing function, as shown in (\ref{eq:cfl-pml-A-cp}). 

In Figures \ref{fig:cfl_q1_const_A} and \ref{fig:cfl_q1_const_B} the 
continuous lines illustrate the theoretical CFL condition for 
the discrete schemes A and B, for different values of the constant 
absorbing function  ($\sigma=1, 10, 25$) and therefore define the 
boundary of the stability region.
In both Figures and through the rest of this section, the markers 
(see the circles, triangles and squares in the plots) are located at 
points  $(h,\Delta t)$ corresponding to the largest value of 
$\Delta t$ for which  the schemes have been checked numerically stable 
in practice. 

In both cases (schemes A and B), the marks lie closed to the boundary of 
the stability region, so the numerical results confirm the predicted 
CFL condition.  For $\sigma=1$, the curve for scheme A is almost the 
same as the one for scheme B, which means that the CFL condition is 
very closed to the one in the physical domain, but this setup corresponds 
to a very low damping which in practice would require a very large 
PML thickness. As soon as the damping factor is large enough, one 
can see that the CFL condition is much more restrictive than the 
standard one.  

\begin{figure}[!ht]
\begin{center}
\begin{minipage}{7.0cm}
\begin{center}
\includegraphics[width=7cm,angle=0]{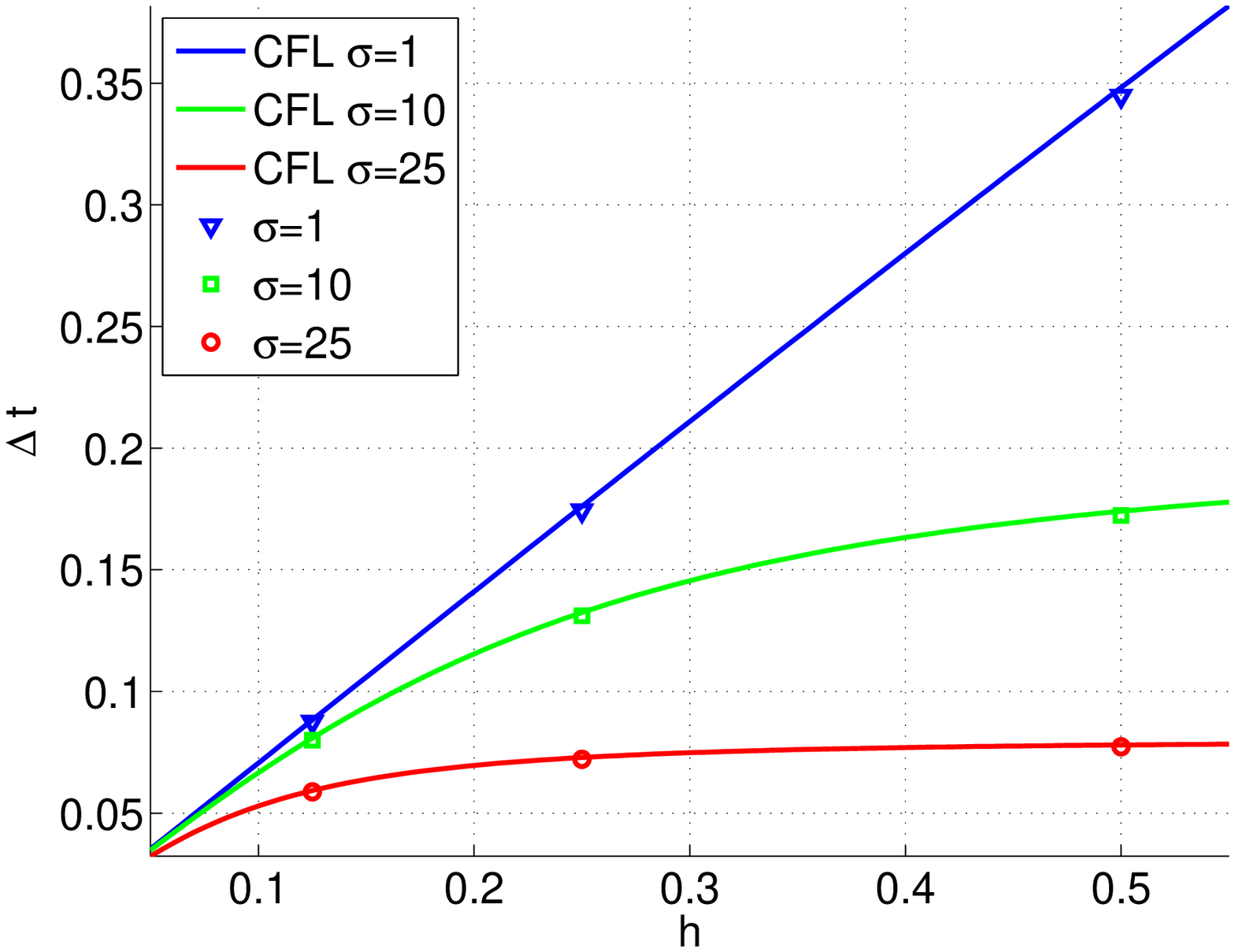} 
\vspace*{-.75cm}
\caption{CFL condition for scheme A with a $Q_1-Q_1^\disc$ 
discretization and a constant absorbing function.}
\label{fig:cfl_q1_const_A}
\end{center}
\end{minipage}
\hspace*{0.75cm}
\begin{minipage}{7.0cm}
\begin{center}
\includegraphics[width=7cm,angle=0]{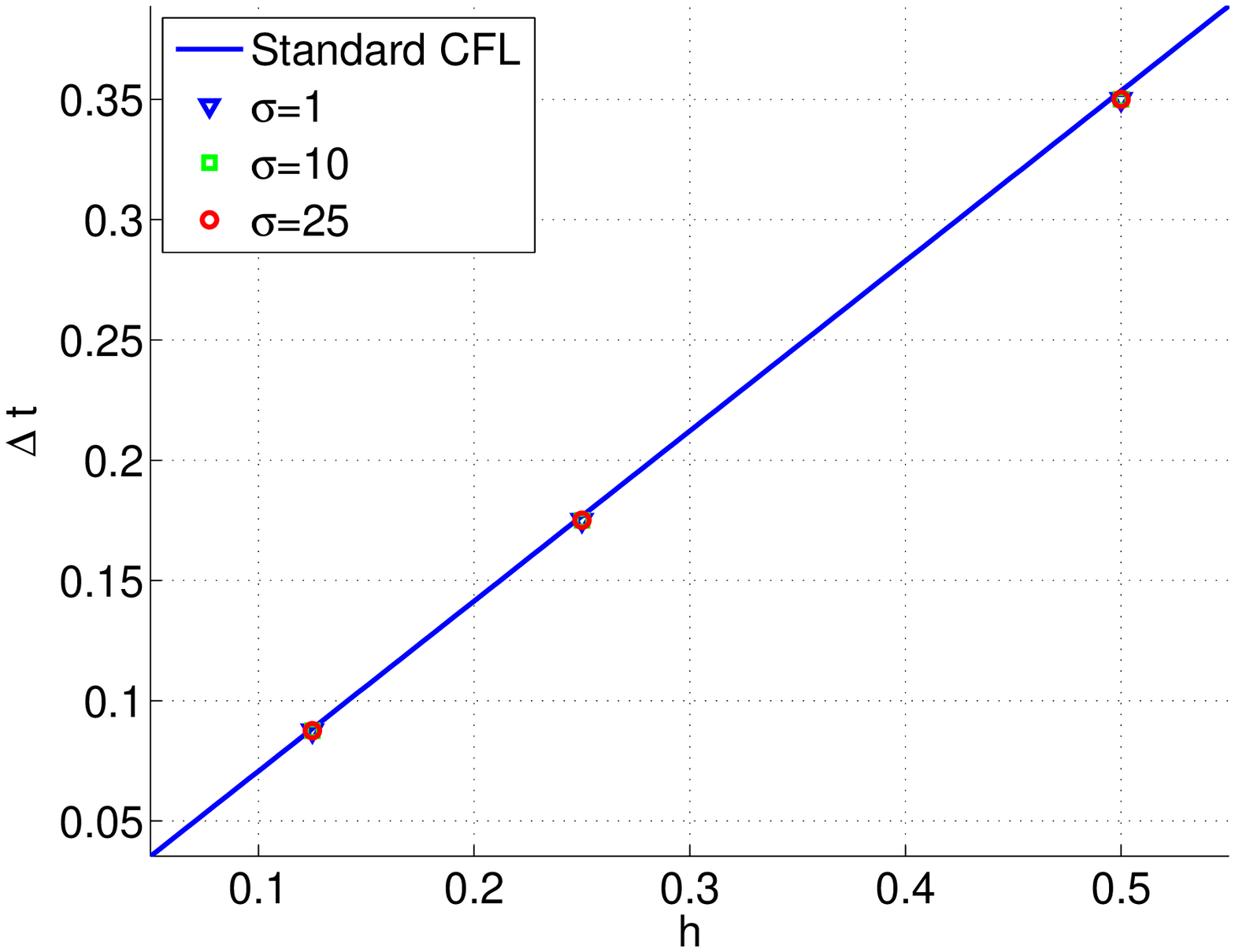} 
\vspace*{-.75cm}
\caption{CFL condition for scheme B with a $Q_1-Q_1^\disc$ 
discretization and a constant absorbing function.}
\label{fig:cfl_q1_const_B}
\end{center}
\end{minipage}
\end{center}
\end{figure}

\paragraph{Higher-order scheme and constant absorbing function.}   
It has been shown in \cite{fauqueux} that, for $r=5$, $cfl_{1,5}=0.1010$. 
According to (\ref{eq:cfl-fluid-3}), we thus have in 2D: 
$cfl_{2,5}=0.1010/\sqrt{2}$.  Therefore, since in our numerical test $c=1$, 
the CFL condition for the scheme A in the corner PML domain with a 
constant absorbing function $\sigma$, is given by  (\ref{eq:cfl-pml-A}) 
with $C=\sqrt{2}/0.1010$. In Figures~\ref{fig:cfl_q5_const_A} and 
\ref{fig:cfl_q5_const_B}, the results are similar to the one obtained 
with $r=1$,  again confirming the theoretical CFL conditions. 
In particular the CFL of scheme B coincides again with the standard one. 
However one notices that the CFL condition  of scheme A is less 
restrictive  than the one obtained with the lower order scheme. 

\begin{figure}[!ht]
\begin{center}
\begin{minipage}{7.0cm}
\begin{center}
\includegraphics[width=7cm,angle=0]{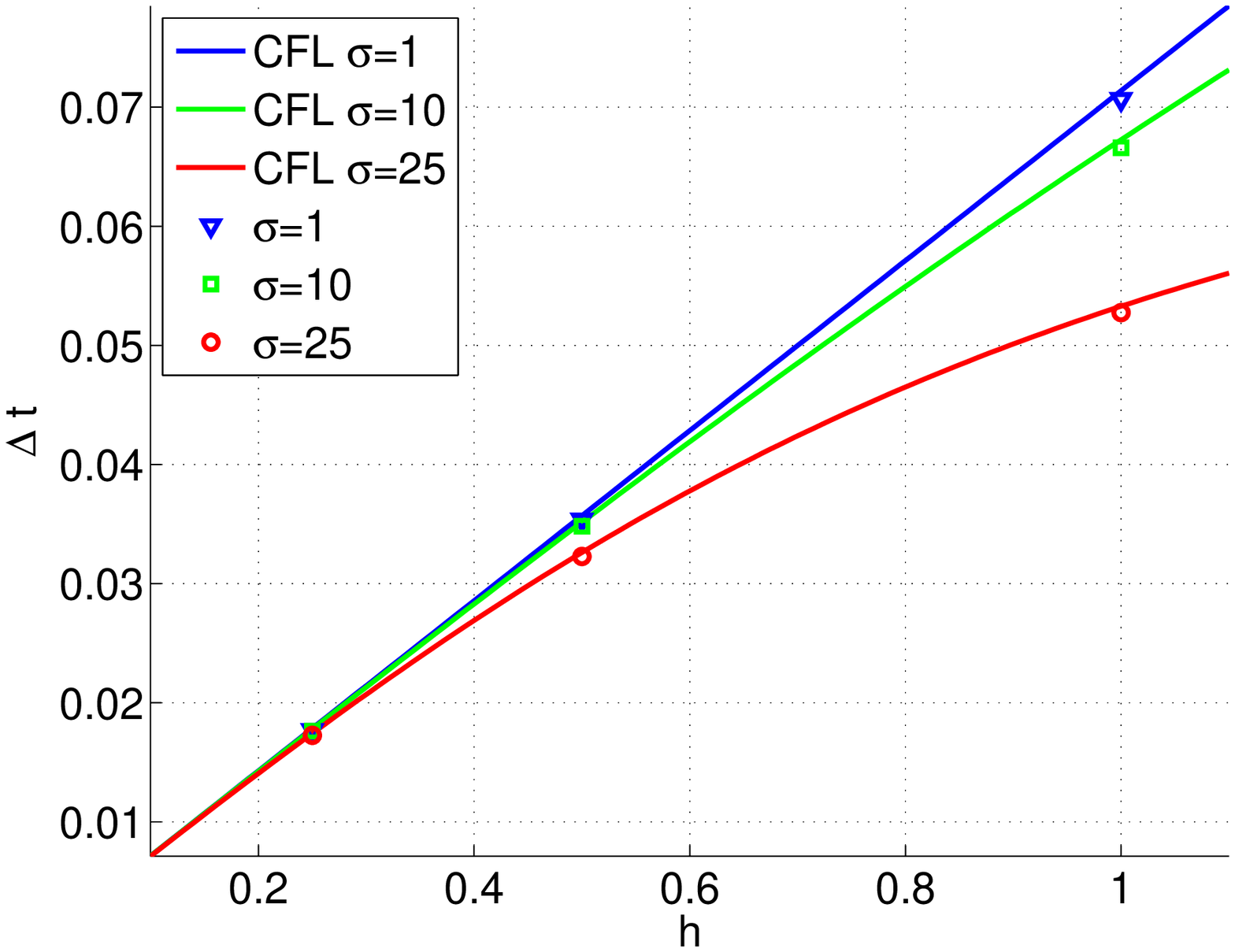} 
\vspace*{-.75cm}
\caption{CFL condition for scheme A with a $Q_5-Q_5^\disc$ 
discretization and a constant absorbing function.}
\label{fig:cfl_q5_const_A}
\end{center}
\end{minipage}
\hspace*{0.75cm}
\begin{minipage}{7.0cm}
\begin{center}
\includegraphics[width=7cm,angle=0]{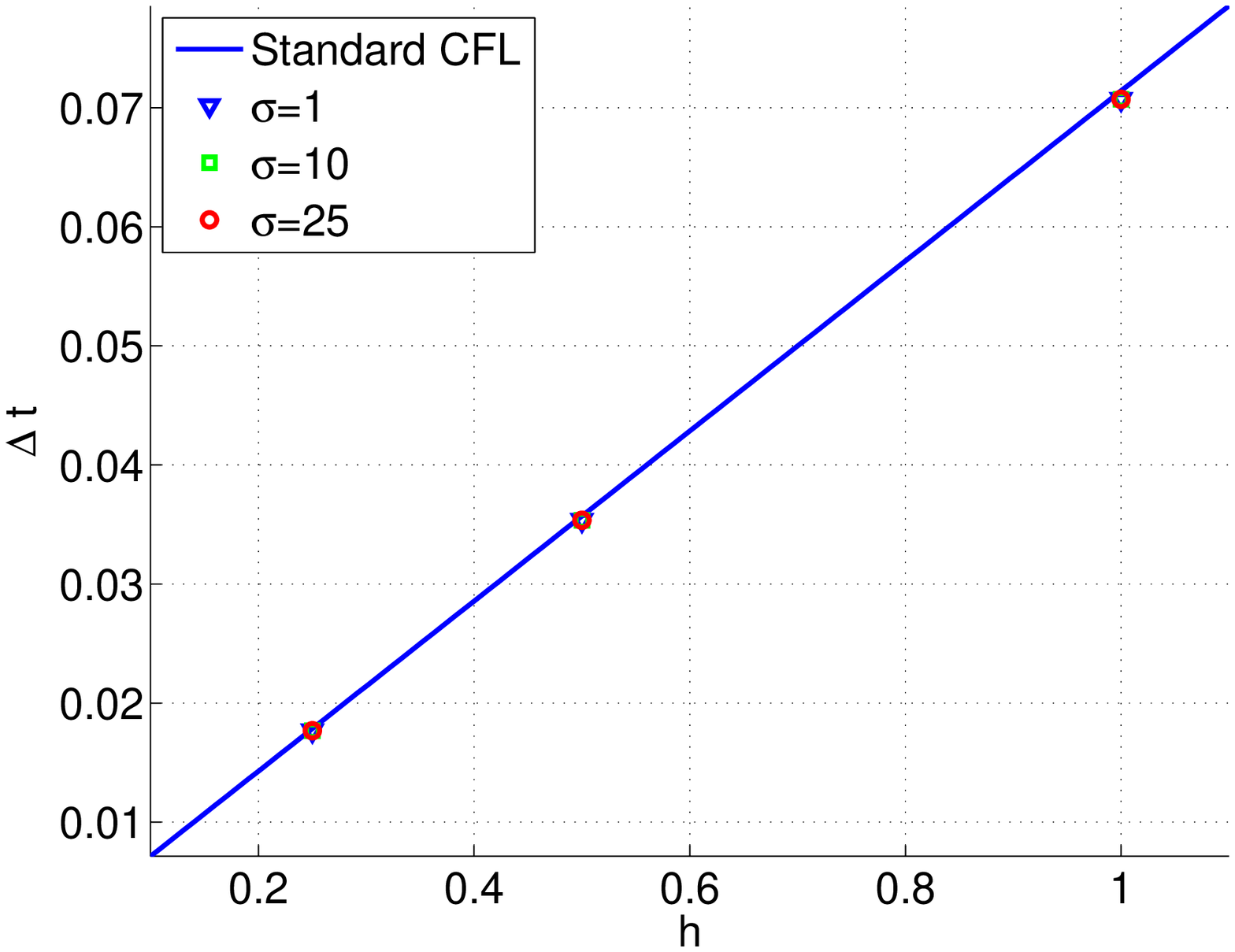} 
\vspace*{-.75cm}
\caption{CFL condition for scheme B with a $Q_5-Q_5^\disc$ 
discretization and a constant absorbing function.}
\label{fig:cfl_q5_const_B}
\end{center}
\end{minipage}
\end{center}
\end{figure}

\paragraph{Quadratic absorbing function.} In this paragraph, we consider 
a continuous quadratic absorbing function whose upper bound is given by 
$\sigma^{*}$. Notice that the theoretical CFL conditions have been 
obtained for a constant damping,  they are therefore not valid anymore 
in this case. However we compare in this paragraph : (i) the stability 
regions obtained with   the theoretical CFL  corresponding to a constant 
damping equal to $\sigma^{*}$ and (ii) the numerical stability region 
for the quadratic profile. 

Figure~\ref{fig:cfl_q1_quadr_B} shows that the CFL condition of 
scheme B, obtained with $Q_1-Q_1^\disc$ finite elements, is still 
independent of $\sigma^{*}$ and thus coincides with the standard 
CFL (\ref{eq:cfl-pml-B-cp}). These numerical results have been checked
also for $Q_5-Q_5^\disc$ finite elements, but since they are analogous
to those shown in Figure~\ref{fig:cfl_q1_quadr_B}, they have not been
included in the plots. This stability behavior, which was not 
guaranteed by the theoretical results, allows us to conjecture 
that the CFL of scheme B  always remains the standard one whatever 
the damping profile is. 
 
Figure \ref{fig:cfl_q1_quadr_A} (resp. \ref{fig:cfl_q5_quadr_A}) 
corresponds to numerical experiments performed with scheme A and $r=1$ 
(resp. $r=5$). In both cases the numerical stability regions still 
depend on the absorbing profile, but this time the markers are no 
longer in the interior of the theoretical stability regions. This could 
be expected, since the theoretical CFL condition has been obtained with 
the maximum value of the damping profile and therefore is more 
restrictive than the actual CFL condition of the scheme. 
As in the constant case,  the numerical CFL condition gets closer to 
the standard one, when higher-order elements are used. In practice 
however, since we cannot predict the CFL condition for a variable 
damping, scheme A is not very convenient to use and its stability 
condition is always more restrictive than the one of scheme B. 

\begin{figure}[!ht]
\begin{center}
\begin{minipage}{14.0cm}
\begin{minipage}{7.0cm}
\begin{center}
\includegraphics[width=7cm,angle=0]{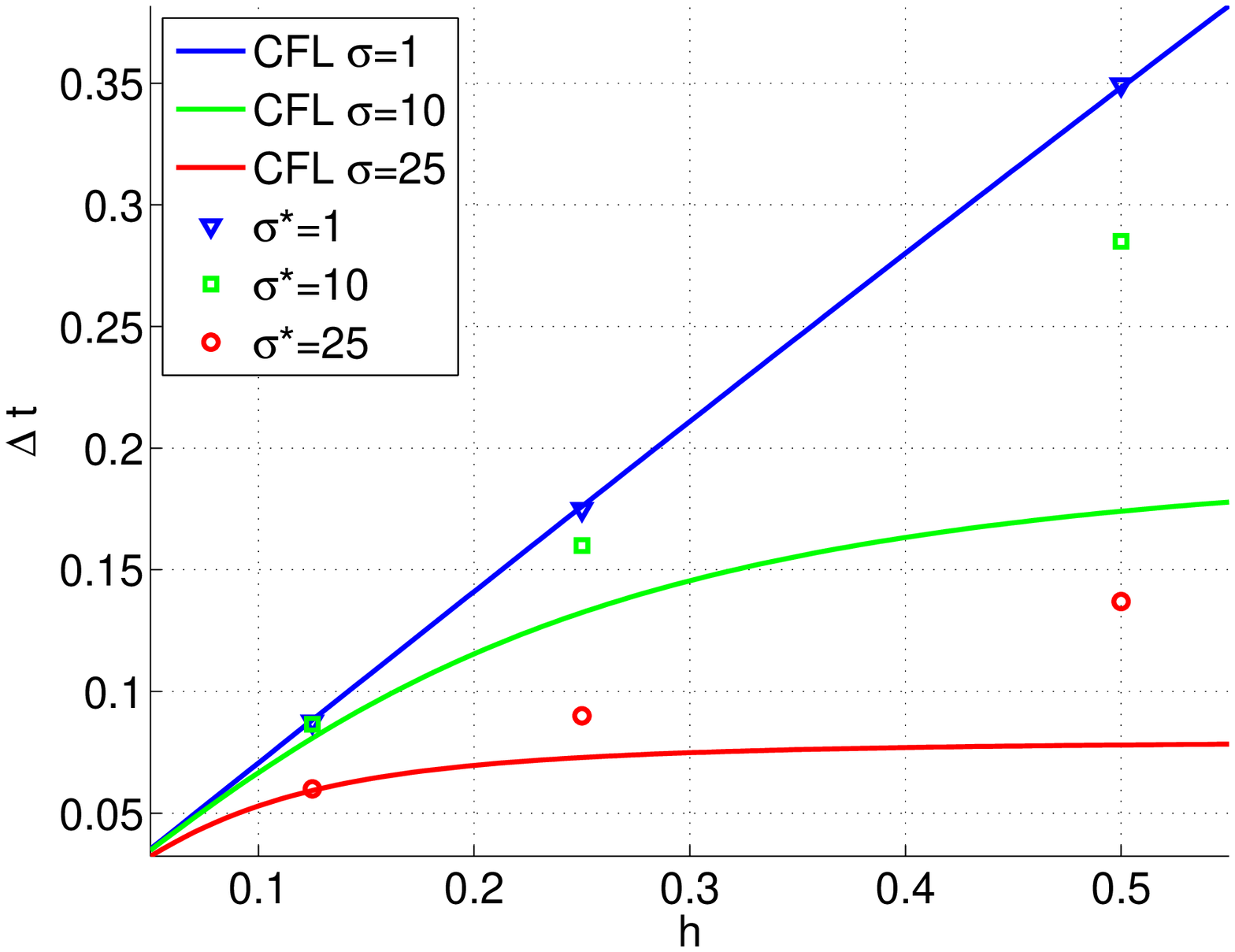} 
\vspace*{-.75cm}
\caption{CFL condition for scheme A with a $Q_1-Q_1^\disc$ 
discretization and a quadratic absorbing function.}
\label{fig:cfl_q1_quadr_A}
\end{center}
\end{minipage}
\hspace*{0.75cm}
\begin{minipage}{7.0cm}
\begin{center}
\includegraphics[width=7cm,angle=0]{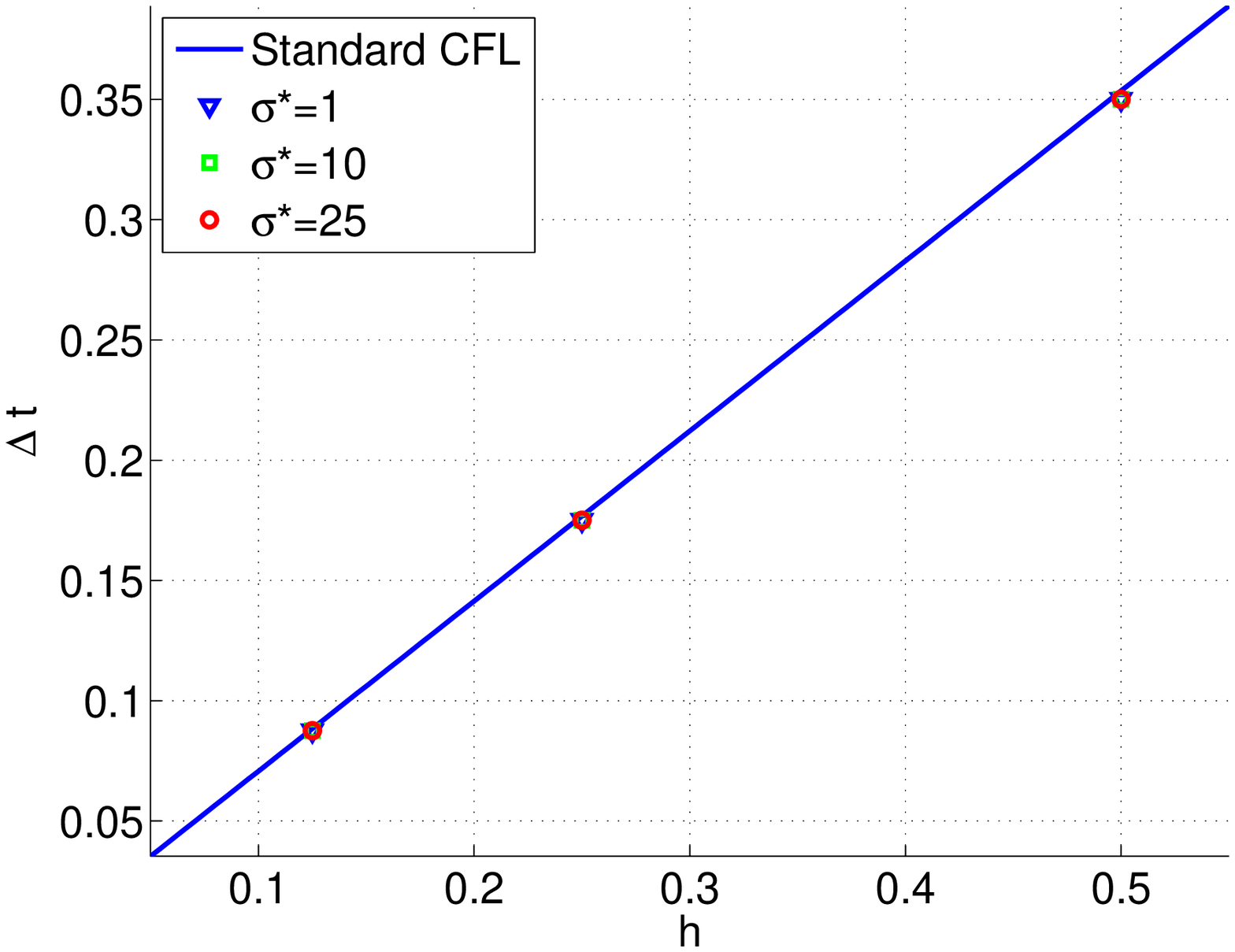} 
\vspace*{-.75cm}
\caption{CFL condition for scheme B with a $Q_1-Q_1^\disc$ 
discretization and a quadratic absorbing function.}
\label{fig:cfl_q1_quadr_B}
\end{center}
\end{minipage}\\
\vspace*{0.5cm}
\begin{center}
\begin{minipage}{7.0cm}
\begin{center}
\includegraphics[width=7cm,angle=0]{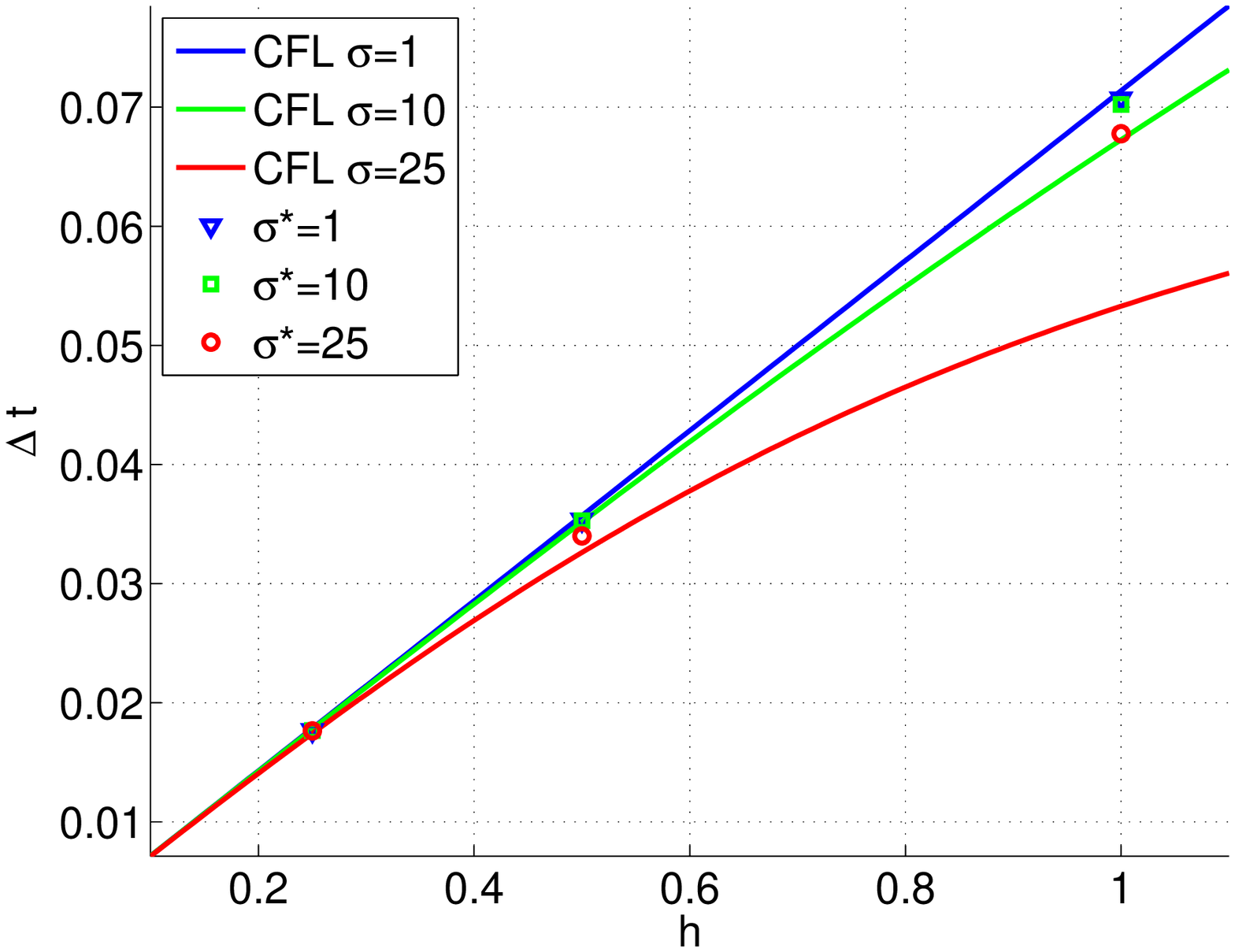} 
\vspace*{-.75cm}
\caption{CFL condition for scheme A with a $Q_5-Q_5^\disc$ 
discretization and a quadratic absorbing function.}
\label{fig:cfl_q5_quadr_A}
\end{center}
\end{minipage}
\end{center}
\end{minipage}
\end{center}
\end{figure}

\paragraph{A comparison between scheme A and scheme B.} 
Finally, to illustrate the qualitative difference between schemes A and 
B, we use a numerical example where scheme B is stable and scheme A is 
unstable. To this purpose, we present some snapshots of two numerical 
simulations at different time steps. We have used a $Q_1-Q_1^\disc$ 
discretization and a uniform grid with spatial sizes $\Delta x=\Delta 
y=0.5$, time step $\Delta t=0.2$. For the construction of the PML, 
we have used a constant absorbing functions with $\sigma=25$. 
In this case, since $\Delta t < h/\sqrt{2}=0.354$, the standard CFL 
holds in the physical domain and the PML corner domain for the scheme B.

\begin{figure}[!ht]
\begin{center}
$\begin{array}{ccc}
\hspace{-2cm}
\psfig{figure=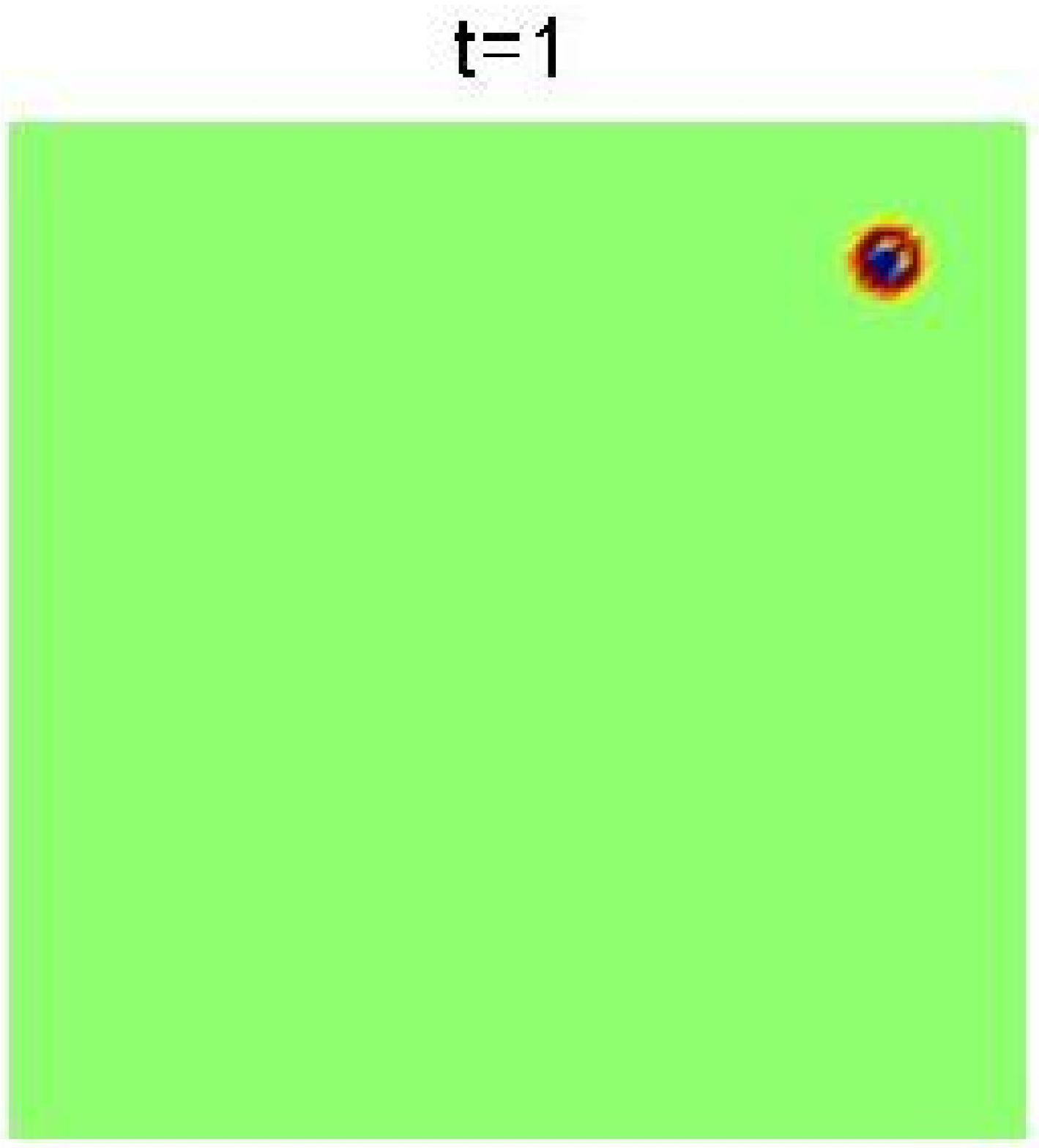,width=6.8cm,angle=0} &
\hspace{-1.5cm}
\psfig{figure=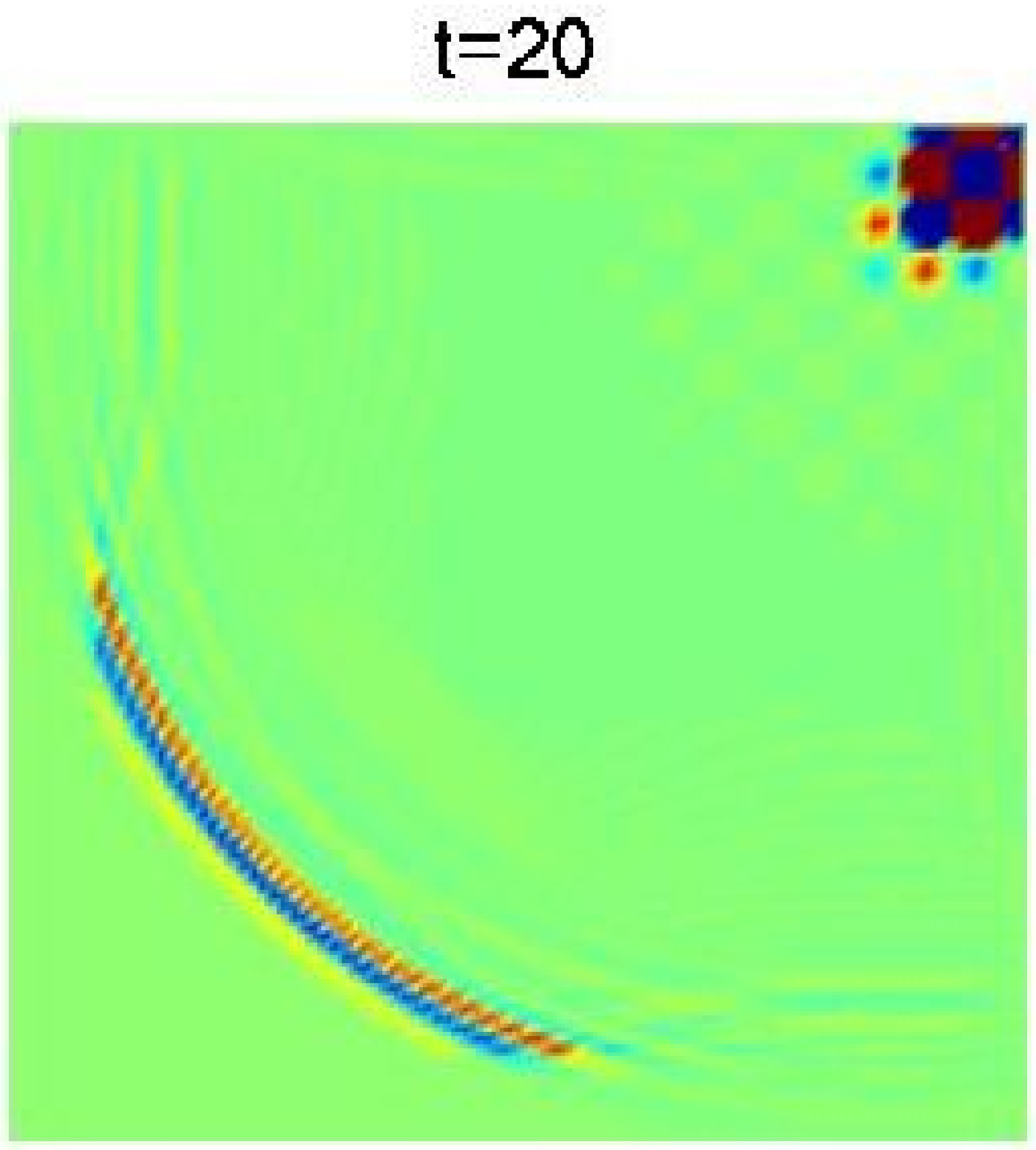,width=6.8cm,angle=0} &
\hspace{-1.5cm}
\psfig{figure=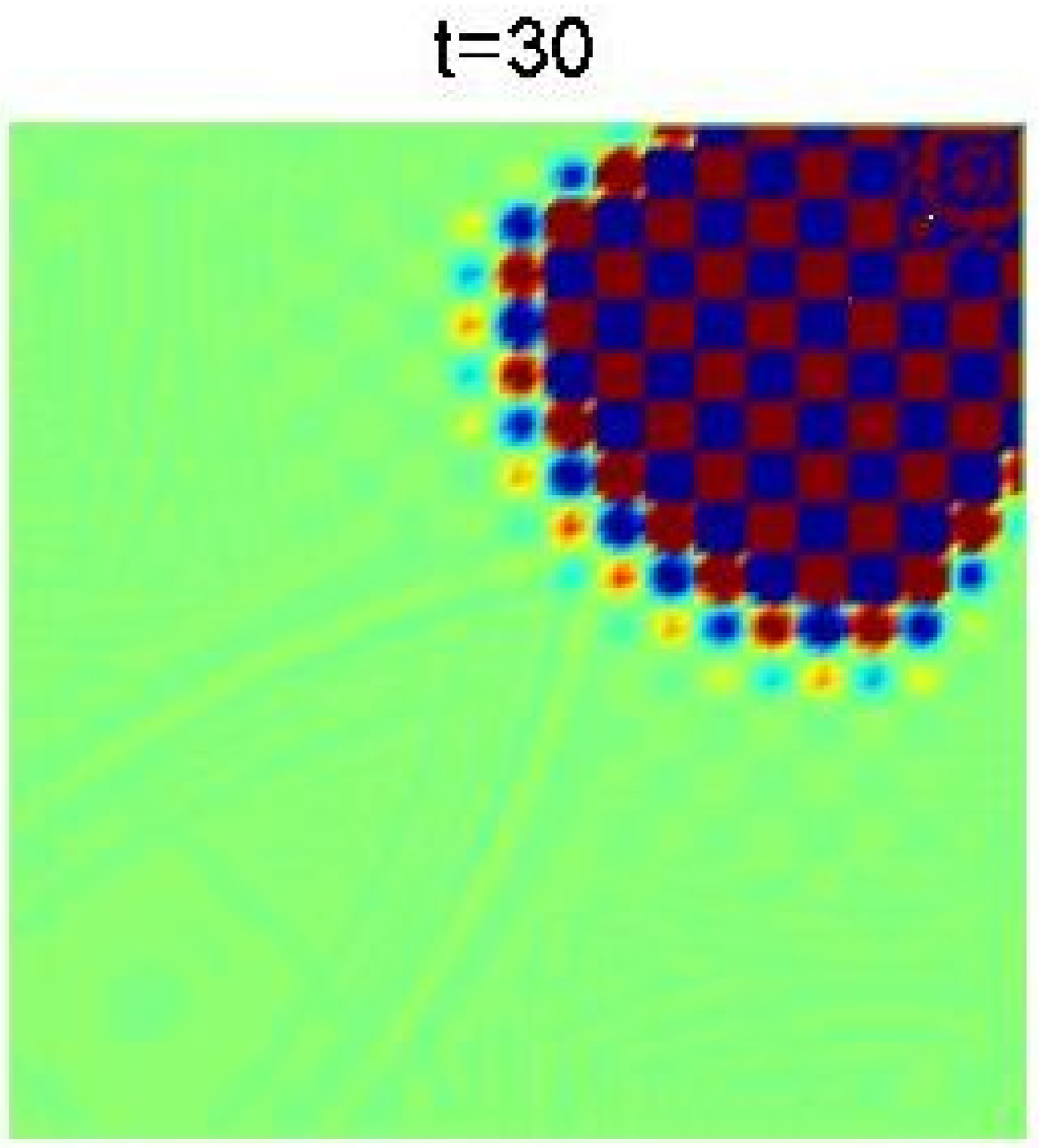,width=6.8cm,angle=0}
\end{array}$
\caption{Snapshots at time steps $n=1,\,20$ and $30$ for the scheme A.\label{fig:stab-A}}
\end{center}
\end{figure}

As expected, whereas the scheme B remains stable (see Figure \ref{fig:stab-B}), 
an instability arises by using the scheme A (see Figure \ref{fig:stab-A}). 
Actually, the more restrictive stability condition for the corner PML, 
which should be used for the scheme A, is here $\Delta t < 0.078$. 
Figure  \ref{fig:stab-A} shows that the instability starts at 
the corner PML which is the closest one to the compact support of the initial 
condition. Moreover this instability arises as soon as the wave 
penetrates the corner PML domain and corresponds to an exponential 
blow up of the numerical solution. 

\begin{figure}[!ht]
\begin{center}
$\begin{array}{ccc}
\hspace{-2cm}\psfig{figure=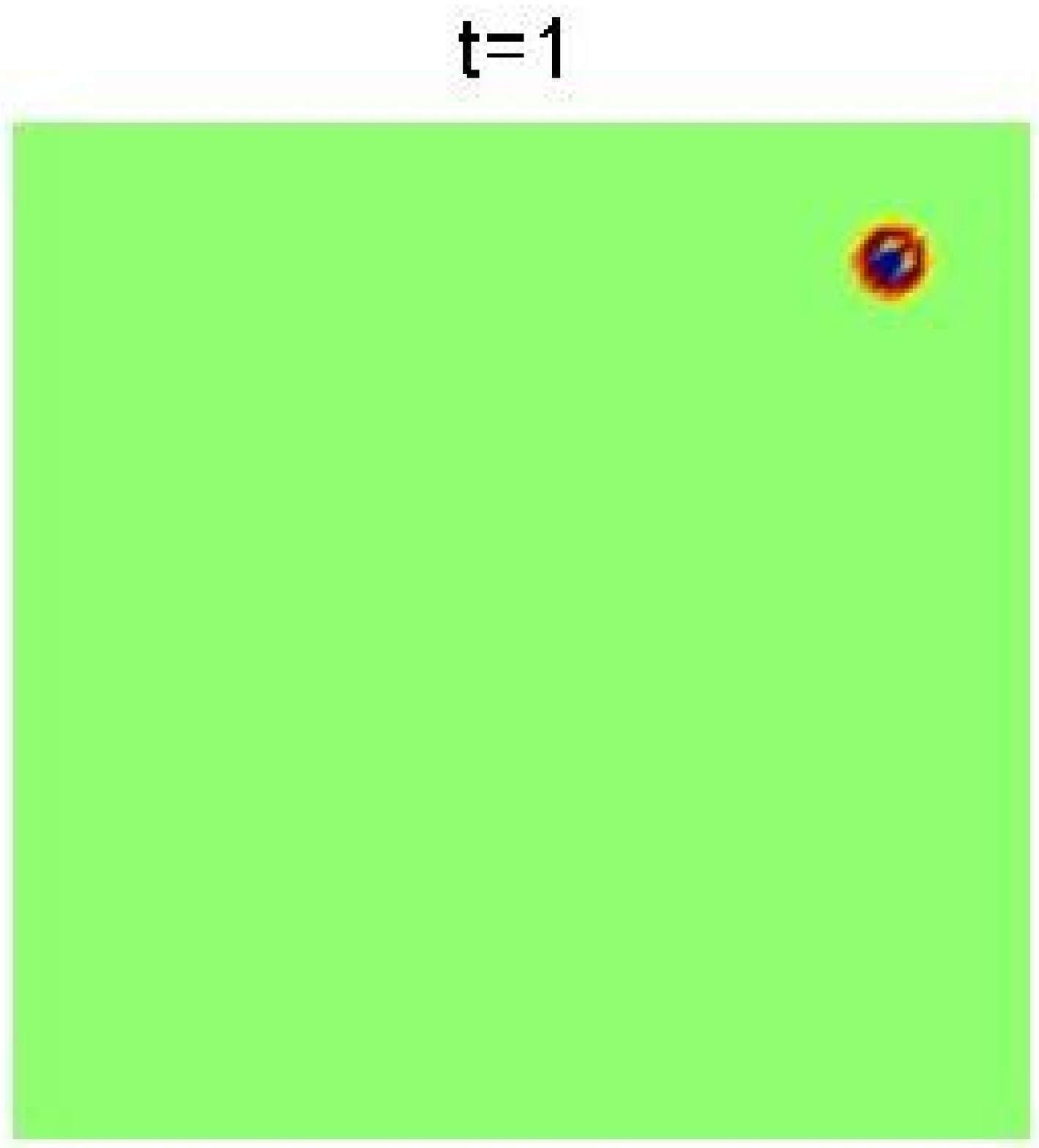,width=6.8cm,angle=0} &
\hspace{-1.5cm}\psfig{figure=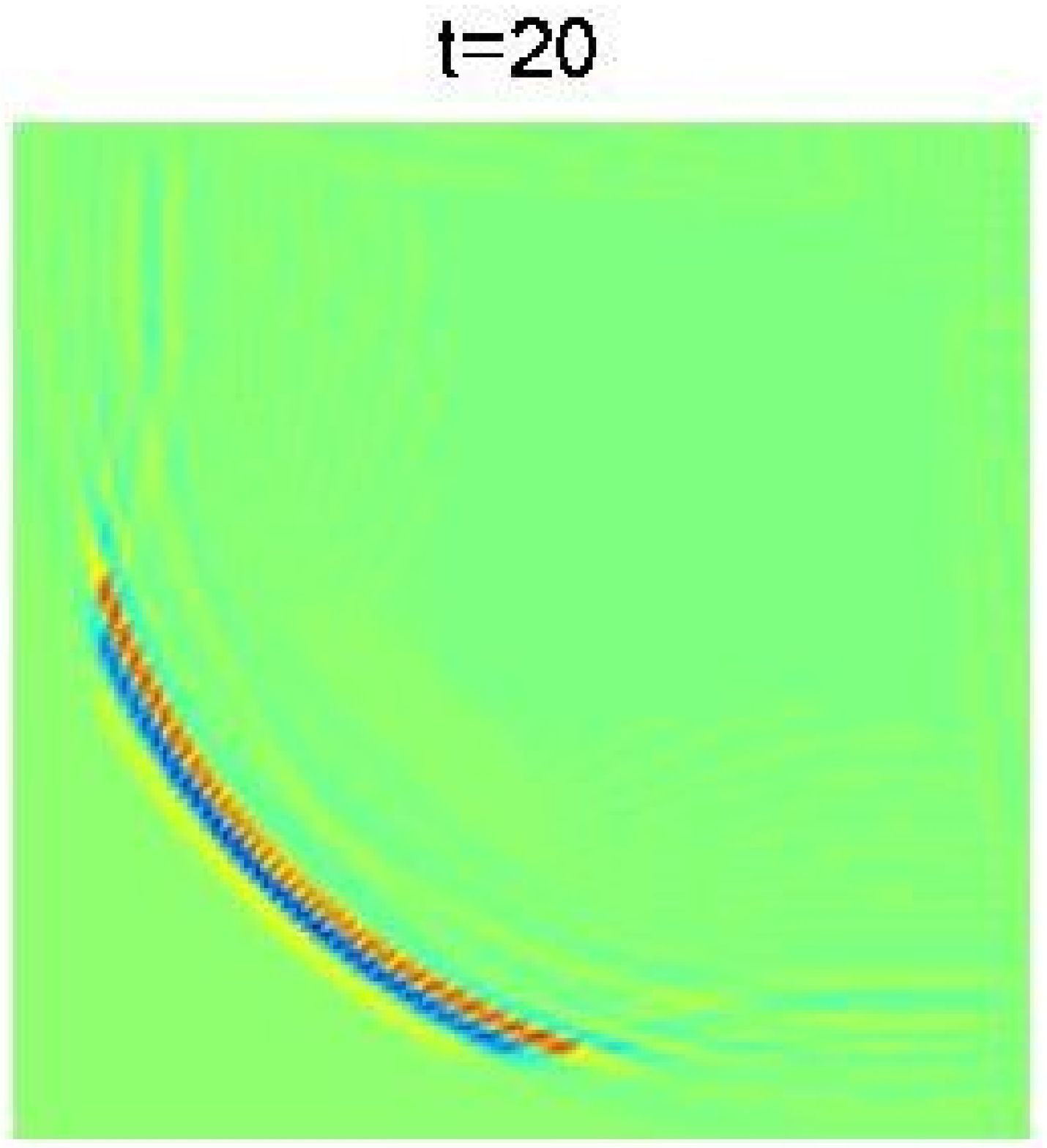,width=6.8cm,angle=0} &
\hspace{-1.5cm}\psfig{figure=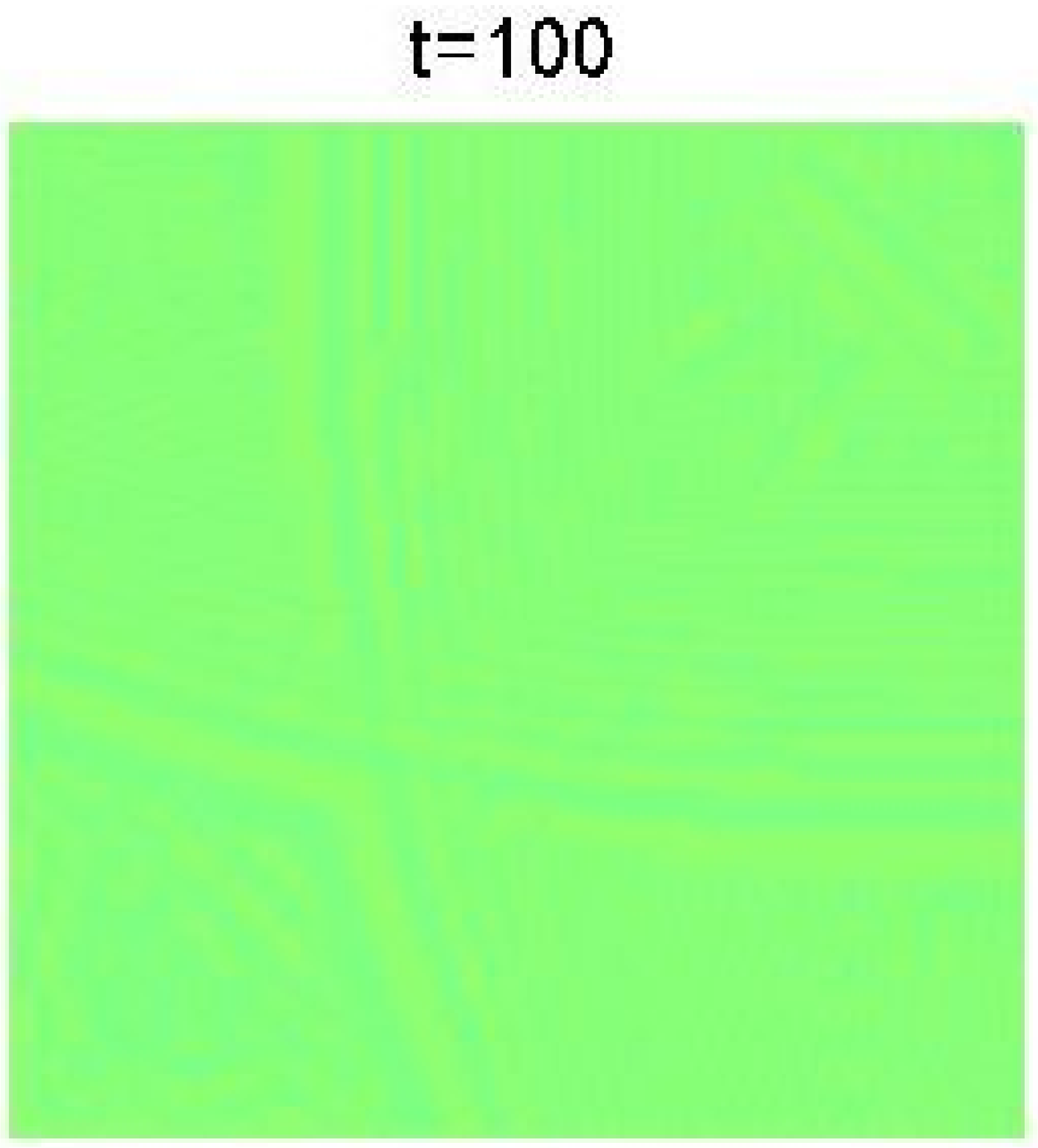,width=6.8cm,angle=0}
\end{array}$
\caption{Snapshots at time steps $t=1,\,20$ and $100$ for the scheme B.\label{fig:stab-B}}
\end{center}
\end{figure}

\section*{Conclusion} 
We have emphasized that the PML models in the Cartesian corner domains 
are always stable and dissipative on the continuous level. However, when
using the unsplit PMLs, one has to be careful on the time discretization
of the equation governing the additional  unknown stated in the corner 
domain. In this work we have shown that a bad choice for this 
discretization leads to a scheme which satisfies a restrictive stability
condition depending on the absorbing function. The numerical solution 
obtained with this scheme blows up exponentially in the corner when one
chooses the maximum time step  allowed by the  stability condition of 
the physical domain. This instability coming from the corner is not due
to a lack of stability of the continuous model and it can be avoided with 
a right choice of the discretization. This new scheme (labeled as ``B'')
leads to a CFL condition independent of the absorbing function and identical
to the stability condition of the discrete scheme in the physical domain.

\bibliography{biblio} 

\end{document}